\documentclass[reqno, oneside, 12pt]{amsart}
\usepackage[letterpaper]{geometry}
\usepackage{amsmath,amssymb}
\usepackage{amsthm}
\usepackage{thmtools}
\usepackage{mathtools}
\usepackage{dsfont}
\usepackage{color}
\usepackage{enumitem}
\usepackage{hyperref}

\usepackage{comment}

\usepackage{tikz}

\usepackage{cases}
\usepackage{multirow}
\usepackage{mathrsfs}

\newcommand{\Z}{{\mathbb{Z}}}
\newcommand{\Q}{{\mathbb{Q}}}
\newcommand{\R}{{\mathbb{R}}}
\newcommand{\C}{{\mathbb{C}}}

\newcommand{\HH}{{\mathbb{H}}}

\newcommand{\Lform}{{\mathcal{L}}}
\newcommand{\LEigenform}{{\tilde{\mathcal{L}}}}

\newcommand{\U}{{\mathcal{U}}}

\DeclareMathOperator{\re}{Re}
\DeclareMathOperator{\im}{Im}

\newcommand{\ep}{\varepsilon}

\newcommand{\defeq}{\vcentcolon=}
\def\SL{{\rm SL}}
\def\GL{{\rm GL}}
\newcommand{\floor}[1]{\left\lfloor #1 \right\rfloor}

\newcommand{\ceil}[1]{\left\lceil #1 \right\rceil}

\renewcommand{\(}{\left(}
\renewcommand{\)}{\right)}
\newcommand{\la}{\left|}
\newcommand{\ra}{\right|}
\newcommand{\Ea}{E_{\mathfrak{a}}}

\newcommand{\Mod}[1]{\ (\mathrm{mod}\ #1)}

\DeclareMathOperator{\ch}{ch}
\DeclareMathOperator{\sh}{sh}
\DeclareMathOperator{\sgn}{sgn}

\newtheorem{theorem}{Theorem}

\newtheorem{lemma}[theorem]{Lemma}
\newtheorem{corollary}[theorem]{Corollary}
\newtheorem{proposition}[theorem]{Proposition}
\newtheorem{conjecture}[theorem]{Conjecture}
\newtheorem{definition}[theorem]{Definition}
\newtheorem{setting}[theorem]{Setting}

\theoremstyle{remark}
\newtheorem*{remark}{Remark}

\numberwithin{equation}{section}
\numberwithin{theorem}{section}
\numberwithin{lemma}{section}
\numberwithin{proposition}{section} 
\numberwithin{example}{section}
\numberwithin{definition}{section}
\numberwithin{corollary}{section}
\numberwithin{setting}{section}
\numberwithin{condition}{section}
\numberwithin{conjecture}{section}

\author{Qihang Sun}
\address{Department of Mathematics, University of Illinois, Urbana, IL 61801}
\email{qihangs2@illinois.edu}

\title[Uniform bounds for Kloosterman sums]{Uniform bounds for Kloosterman sums of half-integral weight, same-sign case}

\date{\today}


\begin{document}

\begin{abstract}
	In the previous paper \cite{QihangFirstAsympt}, the author proved a uniform bound for sums of half-integral weight Kloosterman sums. This bound was applied to prove an exact formula for partitions of rank modulo 3. That uniform estimate provides a more precise bound for a certain class of multipliers compared to the 1983 result by Goldfeld and Sarnak and generalizes the 2009 result from Sarnak and Tsimerman to the half-integral weight case. However, the author only considered the case when the parameters satisfied $\tilde m\tilde n<0$. In this paper, we prove the same uniform bound when $\tilde m\tilde n>0$ for further applications. 
\end{abstract}
\maketitle


\section{Introduction}

For positive integers $N$, $c$ and $m,n\in \Z$, we define the generalized Kloosterman sums with a multiplier system $\nu$ as
\[S(m,n,c,\nu)=\sum_{\substack{0\leq a,d<c \\ \gamma=\begin{psmallmatrix} a&b\\ c&d  		\end{psmallmatrix} 		\in \Gamma}} \overline{\nu}(\gamma)\(\frac{\tilde ma+\tilde nd}c\)\]
where $\Gamma=\Gamma_0(N)$ is a congruence subgroup of $\SL_2(\Z)$, $\nu$ is a weight $k\in \R$ multiplier system on $\Gamma$, and $\tilde n=n-\alpha_{\nu,\infty}$ as defined in \eqref{Alpha_nu}. These Kloosterman sums have been studied by Goldfeld and Sarnak \cite{gs} and Pribitkin \cite{pribitkin}. In a previous paper \cite{QihangFirstAsympt}, the author proved a uniform bound for the sums of such Kloosterman sums and applied the bound to estimate the partial sums of Rademacher-type exact formulas. For example, the Fourier coefficient $G(n)$ of $\gamma(q)$, which is a sixth order mock theta function defined in \cite[(0.17)]{AndrewsMockThetaSixthOrder}, can be written as \cite[Theorem~2.2]{QihangFirstAsympt}
\begin{equation}\label{RademacherExactFormula}
	G(n)=A\(\frac13;n\)=\frac{2\pi\,e(-\frac18)}{(24n-1)^{\frac14}}\sum_{3|c>0}\frac{S(0,n,c,\,(\frac \cdot3)\overline{\nu_\eta})}{c}I_{\frac12}\(\frac{\pi\sqrt{24n-1}}{6c}\),
\end{equation}
where $\nu_\eta$ is the multiplier system of weight $\frac12$ for Dedekind's eta-function (see \eqref{etaMultiplier}) and $I_\kappa$ is the $I$-Bessel function. The author bounded the error
\begin{equation}\label{errorp}
	R_3(n,x)=\frac{2\pi\,e(-\frac18)}{(24n-1)^{\frac14}}\sum_{3|c>x}\frac{S(0,n,c,\,(\frac \cdot3)\overline{\nu_\eta})}{c}I_{\frac12}\(\frac{\pi\sqrt{24n-1}}{6c}\)
\end{equation}
with \cite[Theorem~1.6]{QihangFirstAsympt} to get 
\[R_3(n,\alpha\sqrt{n})\ll_{\alpha,\ep} n^{-\frac1{147}+\ep}.\]

Our estimate can be applied to sums of Kloosterman sum with a class of multiplier systems. However, in the former paper \cite{QihangFirstAsympt} we only focus on the case $\tilde m\tilde n<0$. In this paper we add the complementary case $\tilde m\tilde n>0$ for further applications like Corollary~\ref{mainCorollary_goodnessHarmonic}.

To state the results, we need to classify our half-integral weight multipliers first:
\begin{definition}\label{Admissibility}
	Let $k=\pm \frac12$ and $\nu'=(\frac{|D|}\cdot)\nu_\theta^{2k}$ where $D$ is some even fundamental discriminant and $\nu_\theta$ is the multiplier for the theta function. 
	We say a weight $k$ multiplier $\nu$ on $\Gamma=\Gamma_0(N)$ is \underline{admissible} if it satisfies the following two conditions:
	\begin{itemize}\label{MultplCond}
		\item[(1)] Level lifting: there exist positive integers $B$ and $M$ such that the map $\mathscr{L}: (\mathscr{L}f)(z)=f(Bz)$ gives: 
		\begin{itemize}
			\item[(i)]	an injection from weight $k$ automorphic eigenforms of the hyperbolic Laplacian $\Delta_k$ on $(\Gamma_0(N),\nu)$ to those on $(\Gamma_0(M),\nu')$ and keeps the eigenvalue;
			\item[(ii)]an injection from weight $k$ holomorphic cusp forms on $(\Gamma_0(N),\nu)$ to weight $k$ holomorphic cusp forms on $(\Gamma_0(M),\nu')$. 
		\end{itemize}
		Here $M$ is a multiple of $4$ and $M$ depends on $B$.
		\item[(2)] Average Weil bound: for $x>y>0$ and $x-y\gg x^{\frac 23}$, we have
		\[\sum_{N|c\,\in[y,x]}\frac{|S(m,n,c,\nu)|}{c}\ll_{N,\nu,\ep} (\sqrt x-\sqrt y)( \tilde{m}  \tilde{n} x)^\ep. \]
	\end{itemize}	
\end{definition}

The author proved \cite[Proposition~5.1]{QihangFirstAsympt} that the following class of multipliers satisfy both the conditions: 
\begin{lemma}
	Let $\nu=(\frac{|D|}\cdot)\nu_\theta$ or $\nu=(\frac{|D|}\cdot)\nu_\eta$ where $D$ is a fundamental discriminant and $\nu_\theta$ and $\nu_\eta$ are the multiplier system for the standard theta function and Dedekind's eta function, respectively (see \eqref{thetaMultiplier} and \eqref{etaMultiplier}). Then both $\nu$ and its conjugate are admissible. 
\end{lemma}

Let $\rho_j(n)$ denote the $n$-th Fourier coefficient of an orthonormal basis $\{v_j(\cdot)\}_j$ of $\LEigenform_{k}(N,\nu)$ (the space of square-integrable eigenforms of the weight $k$ Laplacian with respect to $(\Gamma_0(N),\nu)$, see Section~2 for details). Here are our theorems:

\begin{theorem}\label{mainThm}
	Suppose $  \tilde{m}>0$, $\tilde{n}> 0$	 and $\nu$ is a weight $k=\pm\frac12$ admissible multiplier on $\Gamma_0(N)$. We have
	\begin{equation}
		\sum_{N|c\leq X} \frac{S(m,n,c,\nu)}{c}=\!\!\!\sum_{r_j\in i(0,\frac14]}\!\!\! \tau_j(m,n)\frac{X^{2s_j-1}}{2s_j-1}
		+O_{\nu,\ep}\(\(A_u(m,n)+X^{\frac16}\) (  \tilde{m}  \tilde{n}X)^\ep\),
	\end{equation}
	where for $B$ and $M$ in Definition~\ref{Admissibility}, we factor $B\tilde \ell=t_\ell u_\ell^2 w_\ell^2$ with $t_\ell$ square-free, $u_\ell | M^\infty$ positive and $(w_\ell,M)=1$ for $\ell\in \{m,n\}$. Here $s_j=\im r_j+\frac12$, 
	\[\tau_j(m,n)=2i^k\overline{\rho_j(m)}\rho_j(n)\pi^{1-2s_j}(4\tilde m \tilde n)^{1-s_j}\frac{\Gamma(s_j+\frac k2)\Gamma(2s_j-1)}{\Gamma(s_j-\frac k2)}\]
	are the coefficients in \cite{gs} (as corrected by \cite[Proposition~7]{AAAlgbraic16}), and
	\begin{align*}A_u(m,n)&\defeq \(\tilde{m}^{\frac{131}{294}}+u_m\)^{\frac18} \(\tilde{n}^{\frac{131}{294}}+u_n\)^{\frac18} (\tilde{m}\tilde{n})^{\frac3{16}}\\
		&\ll (\tilde{m}  \tilde{n})^{\frac{143}{588}}+\tilde{m}^{\frac{143}{588}}\tilde{n}^{\frac 3{16}}\,u_n^{\frac18}+\tilde{m}^{\frac 3{16}}\tilde{n}^{\frac{143}{588}}\,u_m^{\frac18}+(\tilde{m}  \tilde{n})^{\frac3{16}}(u_mu_n)^{\frac18}.
	\end{align*}
\end{theorem}

\begin{remark}
	We have the following notes for this theorem: 
	\begin{itemize}
		\item The notation $u|M^\infty$ means $u|M^C$ for some positive integer $C$. 
		\item When $u_m$ and $u_n$ are both $O_{N,\nu}(1)$, $A_u(m,n)\ll_{N,\nu} (\tilde m\tilde n)^{\frac{143}{588}}$. 
		\item In general, $A_u(m,n)\ll_{N,\nu} (\tilde m\tilde n)^{\frac14}$. 
		\item When $k=-\frac12$ and $r_j=\frac i4$, we have $\tau_j(m,n)=0$ (see \eqref{CuspFormR0} and \eqref{Coeffi CuspFormR0}). 
		\item The theorem also applies to the case $\tilde m<0$ and $\tilde n<0$ because of \eqref{KlstmSumConj} by conjugation. 
	\end{itemize}
	
\end{remark}

Based on Theorem~\ref{mainThm} and 
\cite[Theorem~1.4]{QihangFirstAsympt}, we are able to reformulate the Linnik-Selberg conjecture in the half-integral weight case: 
\begin{conjecture}
	Suppose $k=\pm\frac 12$ and $\nu$ is a weight $k$ multiplier system on $\Gamma_0(N)$. If there is no eigenvalue for the hyperbolic Laplacian $\Delta_k$ in $(\frac 3{16},\frac14)$, then 
	\begin{equation}
		\sum_{N|c\leq X}\frac{S(m,n,c,\nu)}c-2X^{\frac12}\sum_{r_j=\frac i4} \tau_j(m,n)\ll_{N,\nu,\ep}|\tilde m\tilde nx|^\ep,
	\end{equation}
	where the second sum runs over normalized eigenforms of $\Delta_k$ with eigenvalue $\frac14+r_j^2=\frac3{16}$. 
\end{conjecture}

We also get the following bound by the properties of Bessel functions. 
\begin{theorem}\label{mainThmLastSec}
	With the same setting as Theorem~\ref{mainThm}, for $\beta=\frac12$ or $\frac 32$, we have
	\begin{equation}\label{mainThmLastSecTheEquation}
		\sum_{N|c> \alpha\sqrt{\tilde{m}\tilde{n}}}\frac{S(m,n,c,\nu)}c \mathscr{B}_{\beta}\(\frac{4\pi \sqrt{\tilde{m}\tilde{n}}}c\)\ll_{\alpha,\nu,\ep}A_u(m,n)(\tilde{m}\tilde{n})^{\ep}
	\end{equation}
except the case when $\tilde m$, $\tilde n$ and $k=\pm\frac12$ are all of the same sign such that $\sum_{r_j=\frac i4}\tau_j(m,n)\neq 0$. 
	Here $\mathscr{B}_\beta$ is the Bessel function $I_\beta$ or $J_\beta$. Note that $A_u(m,n)\ll_{N,\nu} (\tilde m\tilde n)^{\frac{143}{588}}$ when $u_m$ and $u_n$ are $O_{N,\nu}(1)$. 
\end{theorem}

\subsection{Application}
In \cite{BringmannOno2012}, Bringmann and Ono constructed Maass-Poincar\'e series to prove that the Fourier coefficients of weight $k\in \Z+\frac12$, $k\leq \frac12$ harmonic Maass forms have exact formulas, i.e. they equal infinite sums of Kloosterman sums (up to Fourier coefficients from holomorphic theta functions when the weight is $\frac12$). The specific case, weight $\frac12$, is crucial as it is related to the rank of partitions. Readers may also refer to \cite{BrmOno2006ivt} or \cite[Section~4]{QihangFirstAsympt} for further introduction. 

We can denote the weight $\frac12$ Maass-Poincar\'e series as $P_{\mathfrak{a}}(z,m,\Gamma_0(N),\frac12,s,\nu)$ for cusp $\mathfrak{a}$ and multiplier $\nu$ of $\Gamma_0(N)$ as \cite[(3.4)]{BringmannOno2012}. The exact formulas are derived from $P_{\mathfrak{a}}$ at $s=\frac 34$, while we only know the convergence at $\re s>1$ (see \cite[Lemma~3.1, (3.7)]{BringmannOno2012}). Bringmann and Ono called a weight $\frac12$ harmonic Maass form on $(\Gamma_0(N),\nu)$ is {\it good} if those Maass-Poincar\'e series corresponding to nontrivial terms in the principal parts of $f$ are individually convergent. They conclude the exact formula for its Fourier coefficients only when $f$ is {\it good}. 

With Theorem~\ref{mainThmLastSec} and \cite[Remark~(1) after Theorem~3.2]{BringmannOno2012}, we have
\begin{corollary}\label{mainCorollary_goodnessHarmonic}
	Suppose $f$ is a weight $\frac12$ harmonic Maass form on $(\Gamma_0(N),\nu)$ and $f$ only has non-trivial principal part at cusp $\infty$. Then $f$ is {\rm good} if $\nu$ is admissible (Definition~\ref{Admissibility}). 
\end{corollary}

The paper is organized as follows. In Section~2 we introduce notations on Kloosterman sums and Maass forms. Section~3 revisits the trace formula by Proskurin \cite{Proskurin1982} and discusses certain properties crucial for the subsequent proof. Section~4 is about bounds on translations of the test function. The proof of Theorem~\ref{mainThm} is presented in Section~5 and is divided into two cases: the weight $k=\frac12$ and $k=-\frac12$. Readers are advised to take particular note of the Remark following Proposition~\ref{AhgDunEsstBd} and be mindful of the weight context in Section~5.

\section{Background}

In this section we recall some basic notions on Maass forms with general weight and multiplier. Let $\Gamma = \Gamma_0(N)$ for some $N\geq 1$ denote our congruence subgroup and $\HH$ denote the upper half complex plane. Fixing the argument in $(-\pi,\pi]$, for any $\gamma\in\SL_2(\R)$ and $z=x+iy\in\HH$, we define 
\[j(\gamma,z)\defeq \frac{cz+d}{|cz+d|}=e^{i\arg(cz+d)} \]
and the weight $k$ slash operator 
\[f|_k\gamma\defeq j(\gamma,z)^{-k}f(\gamma z) \]
for $k\in \R$. 
We say that $\nu:\Gamma\to \C^\times$ is a multiplier system of weight $k$ if
\begin{enumerate}[label=(\roman*)]
	\item $|\nu|=1$,
	\item $\nu(-I)=e^{-\pi i k}$, and
	\item $\nu(\gamma_1 \gamma_2) =w_k(\gamma_1,\gamma_2)\nu(\gamma_1)\nu(\gamma_2)$ for all $\gamma_1,\gamma_2\in \Gamma$, where
	\[w_k(\gamma_1,\gamma_2)\defeq j(\gamma_2,z)^k j(\gamma_1,\gamma_2 z)^k j(\gamma_1\gamma_2, z)^{-k}.\]
\end{enumerate}
If $\nu$ is a multiplier system of weight $k$, then it is also a multiplier system of weight $k\pm2$ and its conjugate $\overline\nu$ is a multiplier system of weight $-k$. 

We are interested in the following multiplier systems of weight $\frac12$ and their conjugates of weight $-\frac12$. The theta-multiplier
$\nu_{\theta}$ on $\Gamma_{0}(4)$ is given by
\begin{equation}\label{thetaMultiplier}
	\theta(\gamma z) = \nu_{\theta}(\gamma) \sqrt{cz+d}\; \theta(z), \quad \gamma=\begin{pmatrix}
		a&b\\c&d
	\end{pmatrix}\in \Gamma_0(4)
\end{equation}
where
\[
\theta(z) \defeq \sum_{n\in\Z} e(n^2 z), \quad \nu_{\theta}(\gamma)=\(\frac cd\)\ep_d^{-1}, \quad \ep_d=\left\{ \begin{array}{ll}
	1&d\equiv 1\Mod 4,\\
	i&d\equiv 3\Mod 4, 
\end{array}\right.
\]
and $\(\frac \cdot\cdot\)$ is the extended Kronecker symbol such that $\(\tfrac{-1}d\)=(-1)^{\frac{d-1}2}=\ep_d^2$ for odd $d\in \Z$.  
The eta-multiplier $\nu_\eta$ on $\SL_2(\Z)$ is given by
\begin{equation}
	\eta(\gamma z) = \nu_{\eta}(\gamma) \sqrt{cz+d}\; \eta(z), \quad \gamma=\begin{pmatrix}
		a&b\\c&d
	\end{pmatrix}\in \SL_2(\Z)
\end{equation}
where
\begin{equation}
	\eta(z) \defeq q^{\frac1{24}}\prod_{n=1}^\infty (1-q^n)\quad \text{and}\quad q=e(z)\defeq e^{2\pi iz}. 
\end{equation}
Explicit formulas for $\nu_\eta$ were given by Rademacher \cite[(74.11), (74.12)]{Rad73Book}: 
\begin{equation}\label{etaMultiplier}
	\nu_\eta(\gamma)=e\(-\frac18\)e^{-\pi i s(d,c)}e\(\frac{a+d}{24c}\), \quad s(d,c)\defeq\sum_{r=1}^{c-1}\frac rc\(\frac{dr}c-\floor{\frac{dr}c}-1\),
\end{equation}
for all $c\in \Z$ and also given by Knopp \cite{Knopp70Book}: 
\begin{equation}
	\nu_{\eta}(\gamma)=\left\{
	\begin{array}{ll}
		\(\dfrac dc\)e\(\dfrac1{24}\Big((a+d)c-bd(c^2-1)-3c\Big)\) & \text{if } c \text{ is odd,}\\
		\(\dfrac cd\)e\(\dfrac1{24}\Big((a+d)c-bd(c^2-1)+3d-3-3cd\Big)\) & \text{if } c \text{ is even,}
	\end{array}
	\right. 
\end{equation}
for $c>0$. The properties $\nu_\eta\(\pm\begin{psmallmatrix}
	1&b\\0&1
\end{psmallmatrix}\)=e\(\tfrac{b}{24}\)$ and $\nu_{\eta}(-\gamma)=i\nu_\eta(\gamma)$ for $c>0$ are convenient in computations.

\subsection{Kloosterman sums}\label{Backgroud 2.2}
For any cusp $\mathfrak{a}$ of $\Gamma=\Gamma_0(N)$, let $\Gamma_{\mathfrak{a}}$ denote its stabilizer in $\Gamma$. For example, $\Gamma_\infty=\{
\pm\begin{psmallmatrix}
	1&b\\0&1
\end{psmallmatrix}:b\in\Z\}$. Let $\sigma_{\mathfrak{a}}\in\SL_2(\R)$ denote its scaling matrix satisfying $\sigma_{\mathfrak{a}}\infty=\mathfrak{a}$ and $\sigma_{\mathfrak{a}}^{-1} \Gamma_{\mathfrak{a}}\sigma_{\mathfrak{a}}=\Gamma_\infty$. We define $\alpha_{\nu,\mathfrak{a}} \in [0,1)$ by the condition
\begin{equation}\label{Alpha_nu}
	\nu\( \sigma_{\mathfrak{a}}\begin{psmallmatrix} 
		1 & 1 \\ 
		0 & 1
	\end{psmallmatrix}\sigma_{\mathfrak{a}}^{-1} \) = e(-\alpha_{\nu,\mathfrak{a}}).
\end{equation}
The cusp $\mathfrak{a}$ is called singular if $\alpha_{\nu,\mathfrak{a}}=0$. When $\mathfrak{a}=\infty$, we drop the subscript, denote $\alpha_{\nu}\defeq\alpha_{\nu,\infty}$, and define $ \tilde{n} \defeq n-\alpha_{\nu}$ for $n\in \Z$.
The Kloosterman sum at the cusp pair $(\infty,\infty)$ with respect to the multiplier system $\nu$ is given by 
\begin{equation}
	\label{eq:kloos_def}
	S(m,n,c,\nu) :=\!\!\! \sum_{\substack{0\leq a,d<c \\ \gamma=\begin{psmallmatrix} a&b\\ c&d 
			\end{psmallmatrix}
			\in \Gamma}}\!\!\!  \overline\nu(\gamma) e\(\frac{ \tilde{m} a+ \tilde{n} d}{c}\)
	=\!\!\! \sum_{\substack{\gamma\in \Gamma_\infty\setminus\Gamma/\Gamma_\infty\\\gamma=\begin{psmallmatrix} a&b\\ c&d 
	\end{psmallmatrix}}} \!\!\!  \overline\nu(\gamma) e\(\frac{ \tilde{m} a+ \tilde{n} d}{c}\).
\end{equation}
They have the relationships
\begin{equation}\label{KlstmSumConj}
	\overline{S(m,n,c,\nu)}=\left\{\begin{array}{ll}
		S(1-m,1-n,c,\overline \nu)&\ \  \text{if\ }\alpha_{\nu}>0,\\
		S(-m,-n,c,\overline\nu)  &\ \  \text{if\ }\alpha_{\nu}=0,
	\end{array}\right. 
\end{equation}
because 
\begin{equation}\label{tildeNConj}
	n_{\overline{\nu}}=\left\{\begin{array}{ll}
		-(1-n)_{\nu}&\ \  \text{if\ } \alpha_{\nu}>0,\\
		n  &\ \ \text{if\ }\alpha_{\nu}=0.
	\end{array}\right. 
\end{equation}

\subsection{Maass forms}
We call a function $f:\HH\rightarrow \C$ automorphic of weight $k$ and multiplier $\nu$ on $\Gamma$ if 
\[f|_k\gamma =\nu(\gamma)f \quad\text{for all }\gamma\in\Gamma.\] Let $\mathcal{A}_k(N,\nu)$ denote the linear space consisting of all such functions and $\Lform_k(N,\nu)\subset \mathcal{A}_k(N,\nu)$ denote the space of square-integrable functions on $\Gamma\setminus\HH$ with respect to the measure \[d\mu(z)=\frac{dxdy}{y^2}\]
and the Petersson inner product 
\[\langle f,g\rangle \defeq\int_{\Gamma\setminus\HH} f(z)\overline{g(z)}\frac{dxdy}{y^2} \]
for $f,g\in\Lform_k(N,\nu)$. For $k\in \R$, the Laplacian
\begin{equation}\label{Laplacian}
	\Delta_k\defeq y^2\(\frac{\partial^2}{\partial x^2}+\frac{\partial^2}{\partial y^2}\)-iky \frac{\partial}{\partial x}
\end{equation}
can be expressed as
\begin{align}
	\Delta_k&=-R_{k-2}L_k-\frac k2\(1-\frac k2\)\\
	&=-L_{k+2}R_k+\frac k2\(1+\frac k2\)
\end{align}
where $R_k$ is the Maass raising operator
\begin{equation}\label{Maass Raising op}
	R_k\defeq \frac k2+2iy\frac{\partial}{\partial z}=\frac k2+iy\(\frac{\partial}{\partial x}-i\frac{\partial}{\partial y}\)
\end{equation}
and $L_k$ is the Maass lowering operator
\begin{equation}\label{Maass Lowering op}
	L_k\defeq \frac k2+2iy\frac{\partial}{\partial \bar z}=\frac k2+iy\(\frac{\partial}{\partial x}+i\frac{\partial}{\partial y}\).
\end{equation}
These operators raise and lower the weight of an automorphic form as
\[(R_k f)|_{k+2}\;\gamma=R_k(f|_{k}\gamma),\quad (L_k f)|_{k-2}\;\gamma=L_k(f|_{k}\gamma),\quad  \text{for\ } f\in\mathcal{A}_k(N,\nu)\]
and satisfy the commutative relations
\begin{equation}\label{RLOperators}
	R_k\Delta_k=\Delta_{k+2}R_k,\quad L_k\Delta_k=\Delta_{k-2}L_k. 
\end{equation}
Moreover, $\Delta_k$ commutes with the weight $k$ slash operator for all $\gamma\in\SL_2(\R)$. 

We call a real analytic function $f:\HH\rightarrow \C$ an eigenfunction of $\Delta_k$ with eigenvalue $\lambda\in\C$ if
\[\Delta_k f+\lambda f=0. \]
From \eqref{RLOperators}, it is clear that an eigenvalue $\lambda$ for the weight $k$ Laplacian is also an eigenvalue for weight $k\pm 2$. 
We call $f\in \mathcal{A}_k(N,\nu)$ a Maass form if it is a smooth eigenfunction of $\Delta_k$ and satisfies the growth condition
\[(f|_k\gamma)(x+iy)\ll y^\sigma+y^{1-\sigma} \]
for all $\gamma\in\SL_2(\Z)$ and some $\sigma$ depending on $\gamma$ when $y\rightarrow +\infty$. Moreover, if a Maass form $f$ satisfies
\[\int_0^1 (f|_k\sigma_{\mathfrak{a}})(x+iy)\;e(\alpha_{\nu,\mathfrak{a}}x)dx=0\]
for all cusps $\mathfrak{a}$ of $\Gamma$, then $f\in\Lform_k(N,\nu)$ and we call $f$ a Maass cusp form. For details see \cite[\S 2.3]{AAimrn}

Let $\mathcal{B}_k(N,\nu)\subset \Lform_k(N,\nu)$ denote the space of smooth functions $f$ such that both $f$ and $\Delta_k f$ are bounded. One can show that $\mathcal{B}_k(N,\nu)$ is dense in $\Lform_k(N,\nu)$ and $\Delta_k$ is self-adjoint on $\mathcal{B}_k(N,\nu)$. If we let $\lambda_0\defeq \lambda_0(k)=\tfrac {|k|}2(1-\tfrac {|k|}2)$, then for $f\in\mathcal{B}_k(N,\nu)$,
\[\langle f,-\Delta_k f\rangle\geq \lambda_0 \langle f,f\rangle,\]
i.e. $-\Delta_k$ is bounded from below. By the Friedrichs extension theorem, $-\Delta_k$ can be extended to a self-adjoint operator on $\Lform_k(N,\nu)$. 
The spectrum of $\Delta_k$ consists of two parts: the continuous spectrum $\lambda\in [\frac14,\infty)$ and the discrete spectrum of finite multiplicity contained in $[\lambda_0,\infty)$.

Let $\lambda_\Delta(G,\nu,k)$ denote the first eigenvalue larger than $\lambda_0$ in the discrete spectrum with respect to the congruence subgroup $G$, weight $k$ and multiplier $\nu$. For weight 0, Selberg \cite{selberg} showed that $\lambda_\Delta(\Gamma(N),\boldsymbol{1},0)\geq \frac3{16}$ for all $N$ and Selberg's famous eigenvalue conjecture states that $\lambda_\Delta(G,\boldsymbol{1},0)\geq\frac14$ for all $G$.  
We introduce the hypothesis $H_\theta$ as
\begin{equation}\label{HTheta}
	H_\theta:\quad \lambda_\Delta(\Gamma_0(N),\boldsymbol{1},0)\geq \tfrac14-\theta^2\ \  \text{for all\ }N. 
\end{equation}
Selberg's conjecture includes $H_0$ and the best progress known today is $H_{\frac 7{64}}$ by \cite{KimSarnak764}. 
We denote $\lambda_\Delta(G,\nu,k)$ as $\lambda_\Delta$ when $(G,\nu,k)$ is clear from context.

Let $\LEigenform_{k}(N,\nu)\subset \Lform_k(N,\nu)$ denote the subspace spanned by eigenfunctions of $\Delta_k$. For each eigenvalue $\lambda$, we write 
\[\lambda=\tfrac14+r^2=s(1-s), \quad s=\tfrac12+ir,\quad r\in i(0,\tfrac14]\cup[0,\infty). \]
So $r\in i\R$ corresponds to $\lambda<\frac14$ and any such eigenvalue $\lambda\in(\lambda_0,\frac 14)$ is called an exceptional eigenvalue. Set
\begin{equation}\label{FirstSpecPara}
	r_\Delta(N,\nu,k)\defeq i\cdot\sqrt{\tfrac14-\lambda_\Delta(\Gamma_0(N),\nu,k)}. 
\end{equation}

Let $\LEigenform_{k}(N,\nu,r)\subset\LEigenform_{k}(N,\nu)$ denote the subspace corresponding to the spectral parameter $r$. Complex conjugation gives an isometry 
\[\LEigenform_{k}(N,\nu,r)\longleftrightarrow\LEigenform_{-k}(N,\overline\nu,r)\]
between normed spaces. 
For each $v\defeq v(z;k)\in \LEigenform_{k}(N,\nu,r)$, we have the Fourier expansion
\begin{equation}\label{Fourier expansion of Maass forms}
	v(z;k)=v(x+iy;k)=c_0(y)+\sum_{ \tilde{n}\neq 0} \rho(n) W_{\frac k2\sgn \tilde{n},\;ir}(4\pi| \tilde{n}|y)e( \tilde{n} x)
\end{equation}
where $W_{\kappa,\mu}$ is the Whittaker function and
\[c_0(y)=\left\{ 
\begin{array}{ll}
	0&\alpha_{\nu}\neq 0,\\
	0&\alpha_{\nu}=0 \text{ and } r\geq 0,\\
	\rho(0)y^{\frac12+ir}+\rho'(0)y^{\frac12-ir}&\alpha_{\nu}=0 \text{ and } r\in i(0,\frac14]. 
\end{array}
\right. 
\]
Using the fact that $W_{\kappa,\mu}$ is a real function when $\kappa$ is real and $\mu\in \R\cup i\R$ \cite[(13.4.4), (13.14.3), (13.14.31)]{dlmf}, if we denote the Fourier coefficient of $v_c\defeq \bar v$ as $\rho_c(n)$, then
\begin{equation}\label{FourierCoeffConj}
	\rho_c(n)=\left\{\begin{array}{ll}
		\overline{\rho(1-n)},&	\alpha_{\nu}>0,\ n\neq 0\\
		\overline{\rho(-n)}, &\alpha_{\nu}=0. 
	\end{array}
	\right.\end{equation}
Moreover, for the Maass operators $R_k$ \eqref{Maass Raising op} and $L_k$ \eqref{Maass Lowering op}, if we define $\lambda(s)\defeq s(1-s)$ for $s\in \C$, then 
\begin{align}
	&\langle R_k v,R_k v\rangle =\(\tfrac 14+r^2+\tfrac{k(2+k)}4\)\langle v,v\rangle,\\
	&\langle L_k v,L_k v\rangle =\(\tfrac 14+r^2-\tfrac{k(2-k)}4\)\langle v,v\rangle. 
\end{align}
Therefore, when $\frac 14+r^2\neq \frac{k(2-k)}4$, the map
\begin{equation}\label{LkBijIso}
	\(\tfrac 14+r^2-\tfrac{k(2-k)}4\)^{-\frac 12}L_k:\ \LEigenform_k(N,\nu,r)\rightarrow \LEigenform_{k-2}(N,\nu,r)
\end{equation}
is a bijective isometry. The reader may refer to \cite[\S4]{DFI2002} for the case when $k\in \Z$, but the calculations for \eqref{LkBijIso} work for $k\in \Z+\frac 12$.

\subsection{Holomorphic cusp forms of half-integral weight}
For positive integers $N,l$ and a weight $k\in \Z+\frac12$ multiplier $\nu$ on $\Gamma_0(N)$, we know that  $\nu$ is also a weight $k+2l$ multiplier system on $\Gamma_0(N)$. For simplicity we denote $K=k+2l\in \Z+\frac12$.  Let $S_{K}(N ,\nu)$ (resp. $M_K(N,\nu)$) be the space of holomorphic cusp forms (resp. modular forms) on $\Gamma_0(N)$ which satisfy the transformation law 
\[F(\gamma z)=\nu(\gamma)(cz+d)^{K} F(z),\quad \gamma\in \Gamma_0(N).\]

The inner product on $S_{K}(N,\nu)$ is defined as
\[\langle F,G\rangle_H\defeq \int_{\Gamma_0(N)\setminus \HH} F(z)\overline{G(z)}y^{K}\frac{dxdy}{y^2},\quad f,g\in S_{K}(N,\nu).  \]
It is known that (see e.g. \cite[\S 5]{rankin_1977_book}) $S_{K}(N,\nu)$ is a finite-dimensional Hilbert space under the inner product $\langle\cdot,\cdot \rangle_H$.
If we take an orthonormal basis $\{F_j(\cdot):\ 1\leq j\leq d\defeq \dim S_K(N,\nu)\}$ of $S_K(N,\nu)$ and write the Fourier expansion of $F_j$ as 
\[F_j(z)=\sum_{n=1}^\infty a_j(n)e(\tilde n z),\] 
then the sum
\begin{equation}\label{Holomorphic forms: Fourier coeff independent of choice}
	\frac{\Gamma(K-1)}{(4\pi\tilde  n)^{K-1}}\sum_{j=1}^d |a_j(n)|^2=1+2\pi i ^{-K}\sum_{N|c}\frac{S(n,n,c,\nu)}c J_{K-1}\(\frac{4\pi \tilde n}c\)
\end{equation}
is independent of the choice of the basis. 

There is a one-to-one correspondence between all $f\in\Lform_k(N,\nu)$ with eigenvalue $\lambda_0$ and weight $k$ holomorphic modular forms $F$ by
\begin{equation}\label{CuspFormR0} 
	f(z)=\left\{\begin{array}{ll}
		y^{\frac k2}F(z)\quad & k\geq 0,\;\  F\in M_k(N,\nu),\\
		y^{-\frac k2}\overline{F(z)}\quad & k<0,\;\  F\in M_{-k}(N,\overline{\nu}). 
	\end{array}
	\right. 	\end{equation}
If we write the Fourier expansion of $f$ as $f(z)=\sum_{n\in \Z} a_y(n)e(\tilde n x)$, then
\begin{equation}\label{Coeffi CuspFormR0} 
	\left\{
	\begin{array}{lll}
		k\geq 0 &\Rightarrow& a_y(n)=0 \text{\ for\ } \tilde n<0,\\
		k<0 &\Rightarrow& a_y(n)=0 \text{\ for\ } \tilde n>0.
	\end{array}
	\right.
\end{equation}

\section{Trace formula}

Let $k\in \Z+ \frac12$, $N$ be a positive integer, and $\mathfrak{a}$ be a singular cusp for the weight $k$ multiplier system $\nu$ on $\Gamma=\Gamma_0(N)$. Define the Eisenstein series associated to $\mathfrak{a}$ by
\begin{equation}\label{EisensteinSeries}
	E_{\mathfrak{a}}(z,s)\defeq\sum_{\gamma\in\Gamma_{\mathfrak{a}}\setminus \Gamma }\overline{\nu(\gamma)w(\sigma_{\mathfrak{a}}^{-1},\gamma)}(\im \sigma_{\mathfrak{a}}^{-1}\gamma z)^s j(\sigma_{\mathfrak{a}}^{-1}\gamma,z)^{-k}.
\end{equation}
For the Maass lowering operator $L_k$ \eqref{Maass Lowering op}, if we let $E_{\mathfrak{a}}'(z,s)$ denote the right hand side of \eqref{EisensteinSeries} with $k$ replaced by $k-2$, then by \cite[(4.48)]{DFI2002} we have
\begin{equation}\label{LkBijEisenstein}
	L_k E_{\mathfrak{a}}(z,s)=(\tfrac k2-s)E_{\mathfrak{a}}'(z,s). 
\end{equation}
Although \cite[\S4]{DFI2002} assumes $k\in \Z$, the verification of \cite[(4.48)]{DFI2002} remains valid for $k\in \Z+\frac 12$. 

We also define the Poincar\'e series for $m>0$ by
\[\U_m(z,s)\defeq\sum_{\gamma\in\Gamma_\infty\setminus \Gamma }\overline\nu(\gamma) (\im \gamma z)^s j(\gamma,z)^{-k}e( \tilde{m} \gamma z).\] 
Both of them are automorphic functions of weight $k$. The properties of these two series can be found in \cite{Proskurin2005}. The Fourier expansion of $\U_m$ is
\begin{equation}\label{FourierExpPcl}
	\U_m(x+iy,s)=y^s e( \tilde{m}z )+y^s\sum_{\ell\in\Z}\sum_{c>0}\frac{S(m,\ell,c,\nu)}{c^{2s}}B(c, \tilde{m}, \tilde{\ell},y,s,k)e( \tilde{\ell} x) 
\end{equation}
where the function $B$ is in \cite[(4.5)]{AAimrn}. When $\re s>1$, $\U_m(\cdot,s)\in \Lform_k(N,\nu)$.
The Fourier expansion of $\Ea$ at $s=\frac12+ir$ for $r\in \R$ is
\begin{equation}
	\begin{alignedat}{2} \label{FourierExpEsst}
		\Ea(x+iy,s)&=\delta_{\mathfrak{a}\infty}y^s & &+\rho_{\mathfrak{a}}(0,r)y^{1-s}+\sum_{\ell\neq 0}\rho_{\mathfrak{a}}(\ell,r)W_{\frac k2\sgn \tilde \ell,\;\frac12-s}(4\pi | \tilde{\ell}|y)e( \tilde{\ell} x)\\
		&=\delta_{\mathfrak{a}\infty}y^s& &+\frac{\delta_{\alpha_{\nu}0}\cdot4^{1-s}\Gamma(2s-1)}{e^{\pi ik/2 
			}\Gamma(s+\frac k2)\Gamma(s-\frac k2)}y^{1-s}\varphi_{\mathfrak{a}0}(s)\\
		&\ & &+\sum_{\ell\neq 0}|\tilde{\ell}|^{s-1} \frac{\pi^s W_{\frac k2\sgn \tilde{\ell},\;\frac12-s}(4\pi | \tilde{\ell}|y)}{e^{\pi ik/2}\Gamma(s+\frac k2 \sgn\tilde{\ell})}\varphi_{\mathfrak{a}\ell}(s)e( \tilde{\ell} x). 
	\end{alignedat}
\end{equation}

Suppose $m$ and $n$ are positive integers and recall the definition of $\tilde n$ in \S\ref{Backgroud 2.2}. The following notations are very important in the remaining part of this paper: 
\begin{setting}\label{conditionATDelta}
	Let $a=4\pi \sqrt{\tilde{m}\tilde{n}}>0$ and $0<T\leq x/3$ with $T\asymp x^{1-\delta}$ where $\delta\in(0,\frac12)$ will be finally taken as $\frac 13$.
\end{setting}  
\begin{setting}\label{conditionphi}
	The test function $\phi\defeq \phi_{a,x,T}:[0,\infty)\rightarrow\R$ is a four times continuously differentiable function which satisfies
	\begin{enumerate} 
		\item $\phi(0)=\phi'(0)=0$ and 
			$\phi^{(j)}(x)\ll_\ep x^{-2+\ep}$ ($j=0,\cdots,4$) as $x\rightarrow\infty$.
		\item $\phi(t)=1$ for $\frac{a}{2x}\leq t\leq \frac{a}{x}$. 
		\item $\phi(t)=0$ for $t\leq \frac{a}{2x+2T}$ and $t\geq \frac a{x-T}$.
		\item $\phi'(t)\ll\(\frac a{x-T}-\frac ax\)^{-1}\ll \frac{x^2}{aT}$.
		\item $\phi$ and $\phi'$ are piecewise monotone on a fixed number of intevals.  
	\end{enumerate}
\end{setting}
This test function was used in \cite{SarnakTsimerman09} and \cite{AAimrn}. We also define the following transformations of $\phi$:  
\begin{equation}\label{widetildePhi}
	\widetilde{\phi}(r)=\int_0^\infty J_{r-1}(u)\phi(u)\frac{du}u
\end{equation}
and for $k\geq 0$,
\begin{equation}\label{widehatPhi}
	\widehat{\phi}(r)\defeq\pi^2 e^{\frac{(1+k)\pi i}2}\frac{\int_0^\infty \(\cos(\frac {k\pi}2+\pi i r)J_{2ir}(u)-\cos(\frac {k\pi}2-\pi i r)J_{-2ir}(u)\)\phi(u)\frac{du}u}{\sh(\pi r)(\ch(2\pi r)+\cos \pi k)\Gamma(\frac12-\frac k2+i r)\Gamma(\frac12-\frac k2-i r)} 
\end{equation}
with the corrected version of \cite[(2.12)]{BlomerSumofHeckeEvaluesOverQP}
\begin{equation}\label{widehatPhi r=i/4}
	\widehat{\phi}\(\tfrac i4\)=\left\{
	\begin{array}{ll}
		e^{\frac{\pi i}4}\int_0^\infty \cos(u)\phi(u) u^{-\frac 32}du&\ \  k=\frac 12, \vspace{4px}\\
		\frac{1}{2}e^{\frac{3\pi i}4}\int_0^\infty \sin(u)\phi(u) u^{-\frac 32}du&\ \  k=\frac 32.
	\end{array}
	\right.
\end{equation}

For an integer $l\geq 1$, let $B_l$ denote an orthonormal basis for the space of holomorphic cusp forms $S_{k+2l}(N,\nu)$ and 
\[\mathscr{B}_k\defeq\bigcup_{l=1}^\infty B_l.\]
Suppose that the Fourier expansion of each $F\in\mathscr{B}_k$ is given by 
\begin{equation}\label{holomphFourierCoeff}
	F(z)\defeq\sum_{n=1}^\infty a_F(n) e(\tilde{n}z). 
\end{equation}
Let $w_F$ denote the weight of $F\in\mathscr{B}_k$. 
Here is the trace formula: 
\begin{theorem}[{\cite[\S6 Theorem]{Proskurin2005}}]\label{traceFormula Theorem}
	Suppose $\nu$ is a multiplier system of weight $k=\frac12$ or $\frac 32$ on $\Gamma$. Let $\{v_j(\cdot)\}$ be an orthonormal basis of $\LEigenform_k(N,\nu)$ and $\Ea(\cdot,s)$ be the Eisenstein series associated to a singular cusp $\mathfrak{a}$. Let $\rho_j(n)$ denote the $n$-th Fourier coefficient of $v_j$. Let $\varphi_{\mathfrak{a}n}(\frac12+ir)$ or $\rho_{\mathfrak{a}}(n,r)$ denote the $n$-th Fourier coefficient of $\Ea(\cdot,\frac12+ir)$ as in \eqref{FourierExpEsst}. Let $\mathscr{B}_k$ and $a_F(n)$ be defined as in \eqref{holomphFourierCoeff}. Then for $\tilde{m}>0$ and $\tilde{n}>0$ we have
	\begin{align}\label{traceFormula}
		\sum_{c>0}\frac{S(m,n,c,\nu)}{c} \phi\(\frac{4\pi\sqrt{ \tilde{m}  \tilde{n}}}{c}\)
		=\mathcal U_k+\mathcal W
		+\!\!\!\sum_{\mathrm{singular\ }\mathfrak{a}}\!\!\!\mathcal{E}_\mathfrak{a},
	\end{align}
	where
	\begin{equation*}
		\mathcal U_k=\sum_{F\in \mathscr{B}_k}\frac{4\Gamma(w_F)e^{\pi i w_F/2}}{(4\pi)^{w_F}(\tilde{m}\tilde{n})^{(w_F-1)/2}}\overline{a_F(m)}a_F(n)\widetilde{\phi}(w_F),
	\end{equation*}
	\begin{equation*}
		\mathcal W=4\sqrt{ \tilde{m}  \tilde{n}} \sum_{r_j}\frac{\overline{\rho_j(m)}\rho_j(n)}{\ch \pi r_j}\widehat{\phi}(r_j),
	\end{equation*}
	and
	\begin{align*}
		\mathcal{E}_\mathfrak{a}&= \int_{-\infty}^\infty\(\frac{ \tilde{m}}{ \tilde{n}}\)^{-ir}\frac{\overline{\varphi_{\mathfrak{a}m}\(\frac12+ir\)}\varphi_{\mathfrak{a}n}\(\frac12+ir\) \widehat{\phi}(r) dr}{\ch \pi r\;|\Gamma(\frac12+\frac k2 +ir)|^2} \\
		&=4\sqrt{ \tilde{m}  \tilde{n}}\cdot\frac{1}{4\pi}\int_{-\infty}^\infty  \overline{\rho_{\mathfrak{a}}\(m,r\)}\rho_{\mathfrak{a}}\(n,r\)\frac{ \widehat{\phi}(r) }{\ch \pi r} dr. 
	\end{align*}
\end{theorem}

\begin{remark}
	We clarify two points in the theorem. 	
	\begin{itemize}		
		\item[(1)] In the term $\mathcal{U}_k$ corresponds to holomorphic cusp forms, each function $F\in \mathscr{B}_k$ has weight $w_F=k+2l\geq \frac 52$.  
		
		\item[(2)] The equality of the two expressions in $\mathcal{E}_\mathfrak{a}$ is by \eqref{FourierExpEsst}: 
		\[\sqrt{\frac{\tilde{n}}{\pi}}\,\rho_\mathfrak{a}(n,r)=
		\frac{e^{-\frac{\pi i k}2} \pi^{ir} \tilde{n}^{ir}}{\Gamma\(\frac12+ir+\frac k2\sgn  \tilde{n}\)}\varphi_{\mathfrak{a}n}\(\frac12+ir\). \]
	\end{itemize}
\end{remark}

\subsection{Properties of admissible multipliers}
Suppose $\nu$ is a weight $k$ admissible multiplier system on $\Gamma=\Gamma_0(N)$ (Definition~\ref{Admissibility}) with parameters $B$, $M$ and $D$. Then we have the following propositions. 
\begin{proposition}\cite[Proposition~5.7]{QihangFirstAsympt}\label{specParaBoundWeightHalf}
	Suppose that $\nu$ satisfies condition (1) of Definition~\ref{Admissibility} and assume $H_\theta$ \eqref{HTheta}. Then we have
	$2\im r_\Delta\(N,\nu,k\)\leq\theta. $
\end{proposition}
\begin{proposition}\label{Level lifting: holo cusp forms, compare}
	Suppose that $\nu$ satisfies condition (1) of Definition~\ref{Admissibility} with $\nu'=(\frac{|D|}\cdot)\nu_\theta^{2k}$. For $l\in \Z$, let $K=k+2l\geq \frac 52$.  Suppose $\{F_{j,l}(\cdot)\}_j$ is an orthonormal basis of $S_{K}(N,\nu)$ and $\{G_{j,l}(\cdot)\}_j$ is an orthonormal basis of $S_{K}(M,\nu')$. Denote $a_{F,j,l}(n)$ as the Fourier coefficient of $F_{j,l}$ and $a_{G,j,l}(n)$ as the Fourier coefficient of $G_{j,l}$. Then we have
	\begin{equation}\label{Level lifting: holo cusp forms, compare, equation}
		\sum_{j=1}^{\dim S_{K}(N,\nu)}|a_{F,j,l}(n)|^2\ll_{N,\nu} \sum_{j=1}^{\dim S_{K}(M,\nu')}|a_{G,j,l}(B\tilde n)|^2. 
	\end{equation}
\end{proposition}
\begin{proof}
	By condition (1) of Definition~\ref{Admissibility}, we know that
	\begin{equation}\label{Holomorphic forms: level lifting}
		\left\{ [\Gamma_0(N):\Gamma_0(M)]^{-\frac12} F_{j,l}(Bz):\ 1\leq j\leq \dim S_{K}(N,\nu)\right\}
	\end{equation}
	is an orthonormal subset of $S_{K}(M,\nu')$.
	Since the left hand side of \eqref{Holomorphic forms: Fourier coeff independent of choice} is independent from the choice of basis, we expand \eqref{Holomorphic forms: level lifting} to an orthonormal basis of $S_{K}(M,\nu')$ and get the result. 
\end{proof}

Now we start to prove a bound for the right hand side of \eqref{Level lifting: holo cusp forms, compare, equation}. First we have
\begin{proposition}[{\cite[Theorem~1]{Waibel2017FourierCO}}]\label{Waibel's bound prop}
	For $K\in \Z+\frac12$, $K\geq \frac 52$ and a quadratic character $\chi$ modulo $M$, suppose 
	\[\left\{\Phi_j=\sum_{n=1}^\infty a_{j}(n)e(nz):\ 1\leq j\leq d\defeq\dim S_K(M,\chi\nu_\theta^{2K})\right\}\]
	is an orthonormal basis of $S_K(M,\chi\nu_\theta^{2K})$. For $n\geq 1$, write $n=tu^2w^2$ with $t$ square-free, $u|M^{\infty}$ and $(w,M)=1$. Then we have
	\[\frac{\Gamma(K-1)}{(4\pi n)^{K-1}}\sum_{j=1}^d|a_j(n)|^2\ll_{K,M,\ep}  \(n^{\frac 37}+u\)n^{\ep}. \]
\end{proposition}

	Note that the implied constant in Waibel's bound depends on $K$ when expressing Bessel functions (see \cite[after (8)]{Waibel2017FourierCO} and \cite[Theorem~1 and p. 400]{Iwaniec1987HalfWeight}). For our proof, it's essential that the bound remains uniform across the weights. We modify the estimate and get the following proposition.

\begin{proposition}\label{Waibel's bound prop modified 143}
	With the same setting as Proposition~\ref{Waibel's bound prop},
	\[\frac{\Gamma(K-1)}{(4\pi n)^{K-1}}\sum_{j=1}^d|a_j(n)|^2\ll_{M,\ep} (n^{\frac{19}{42}}+u)n^{\ep}. \]
\end{proposition}
\begin{proof}
	We do the same preparation as \cite[around (8)]{Waibel2017FourierCO} to estimate the right hand side of \eqref{Holomorphic forms: Fourier coeff independent of choice}. Let $P>1+(\log 2nM)^2$ (finally chosen to be $\asymp _M n^{\frac 17}$) and define the set of prime numbers
	\[\mathcal{P}\defeq\ \{p \text{ prime}: P<p\leq 2P,\ p\nmid 2nM\}. \]
	Here we have $\#\mathcal P\asymp P/\log P$. 
	
	For $\{\Phi_j\}_j$ a orthonormal basis of $S_{K}(M,\chi\nu_\theta^{2K})$, the set $\{[\Gamma_0(M):\Gamma_0(pM)]^{-\frac12} \Phi_j\}_j$ is an orthonormal subset of $S_{K}(pM,\chi\nu_\theta^{2K})$. Recall \eqref{Holomorphic forms: Fourier coeff independent of choice} and we have
	\begin{equation}\label{Holo level average}
		 \frac{\Gamma(K-1)}{(4\pi n)^{K-1}}\sum_{j=1}^d \frac{|a_j(n)|^2}{[\Gamma_0(M):\Gamma_0(pM)]}\leq 1+2\pi i ^{-K}\sum_{pM|c}\frac{S(n,n,c,\chi\nu_\theta^{2K})}c J_{K-1}\(\frac{4\pi n}c\).  
	\end{equation}
	For those $p\in \mathcal P$, $[\Gamma_0(M):\Gamma_0(pM)]\leq p+1$. Summing \eqref{Holo level average} on $p\in \mathcal P$ and dividing $\#\mathcal P$ we get
	\begin{equation}\label{Holo level average ed}
		\frac{\Gamma(K-1)}{(4\pi n)^{K-1}}\sum_{j=1}^d |a_j(n)|^2\ll P+(\log P)\sum_{p\in \mathcal P}\left|\sum_{pM|c}  \frac{S(n,n,c,\chi\nu_\theta^{2K})}c J_{K-1}\(\frac{4\pi n}c\)\right|.  
	\end{equation}
	The average estimate of 
	\[K_{pM}^{(\mu)}(x)\defeq \sum_{pM|c\leq x}\frac{S(n,n,c,\chi\nu_{\theta}^{2K})}{\sqrt c}e\(\frac{2\mu n}c\),\quad \mu\in \{-1,0,1\}\]
	can be found in \cite[(19)]{Waibel2017FourierCO} that for $\mu\in \{-1,0,1\}$, 
	\begin{equation}\label{KQmu general case}
		\sum_{p\in \mathcal P}|K_{pM}^{(\mu)}(x)|\ll_{M,\ep}\(xun^{-\frac12}+xP^{-\frac12}+(x+n)^{\frac 58}\(x^{\frac 14}P^{\frac 38}+n^{\frac 18}x^{\frac 18}P^{\frac 14}\)\)(nx)^\ep. 
	\end{equation}
	We break the sum on $c\equiv 0\Mod{pM}$ at the right hand side of \eqref{Holo level average ed} into $c\leq n$ and $c\geq n$ to estimate. The uniform bound of $J$-Bessel functions is given by \cite{landau_2000}
	\begin{equation}\label{Landau JBessel explicit}
		|J_\beta(x)|\leq c_0x^{-\frac13}\quad \text{for all }\beta>0\text{ and }x>0, 
	\end{equation}
where $c_0=0.7857\cdots$. 

When $c\leq n$, using \eqref{Landau JBessel explicit} and \cite[(10.6.1)]{dlmf}
\[2J_{\beta-1}'(x)=J_{\beta-2}(x)-J_{\beta}(x),\]
we find that
\begin{equation}\label{Bessel derivative c small}
	\(x^{-\frac12}J_{K-1}\(\frac{4\pi n}x\)\)'\ll n^{-\frac13}x^{-\frac 76}+n^{\frac 23}x^{-\frac {13}6}. 
\end{equation}
Then a partial summation using \eqref{KQmu general case}, \eqref{Bessel derivative c small} and \eqref{Landau JBessel explicit} yields
\begin{equation}\label{Holo c small}
	\sum_{p\in \mathcal P}\left|\sum_{pM|c\leq n}  \frac{S(n,n,c,\chi\nu_\theta^{2K})}c J_{K-1}\(\frac{4\pi n}c\)\right| \ll_{M,\ep} (n^{\frac {19}{42}}+u)n^\ep. 
\end{equation}

When $c\geq n$, we get another bound 
\begin{equation}\label{Bessel derivative c large}
	\(x^{-\frac12}J_{K-1}\(\frac{4\pi n}x\)\)'\ll nx^{-\frac 52}\quad \text{for }x\geq n
\end{equation}
by $|J_{\beta-1}(x)|\leq \frac{(x/2)^{\beta-1}}{\Gamma(\beta)}$ \cite[(10.14.4)]{dlmf} and $|J_\beta(x)|\leq 1$ \cite[(10.14.1)]{dlmf}. Remember $K\geq \frac 52$ here. We do a partial summation again using \eqref{KQmu general case} and \eqref{Bessel derivative c large} and get
\begin{equation}\label{Holo c large}
	\sum_{p\in \mathcal P}\left|\sum_{pM|c\geq n}  \frac{S(n,n,c,\chi\nu_\theta^{2K})}c J_{K-1}\(\frac{4\pi n}c\)\right| \ll_{M,\ep} (n^{\frac {3}{7}}+u)n^\ep. 
\end{equation}
From \eqref{Holo c small}, \eqref{Holo c large}, \eqref{Holo level average ed} and $P\asymp_M n^{\frac 17}$, we finish the proof. 
\end{proof}

Combining Proposition~\ref{Level lifting: holo cusp forms, compare} and Proposition~\ref{Waibel's bound prop modified 143}, one observes the following bound: 
\begin{proposition}\label{Level lifting: holo cusp forms, final bound}
	With the same setting as Proposition~\ref{Level lifting: holo cusp forms, compare}, we factor $B\tilde n=t_n u_n^2 w_n^2$ with $t_n$ square-free, $u_n|M^{\infty}$ and $(w_n,M)=1$. Then 
	\[\frac{\Gamma(K-1)}{(4\pi \tilde n)^{K-1}}\sum_{j=1}^{\dim S_{K}(N,\nu)}|a_{F,j,l}(n)|^2\ll_{N,\nu,\ep}(\tilde n^{\frac{19}{42}}+u_n)\tilde n^{\ep}.  \]
\end{proposition}

\section{Bounds on $\widetilde{\phi}$ and $\widehat{\phi}$ }
In this section, all of the implied constants among the estimates for $\widetilde{\phi}$ and $\widehat{\phi}$ depend on $N$ and the multiplier system $\nu$ unless specified. Recall the definitions \eqref{widetildePhi} and \eqref{widehatPhi}. 
To deal with the $\Gamma$-function in the denominator of $\widehat{\phi}$, we need \cite[(5.6.6-7)]{dlmf}
\begin{equation}\label{GammaxiyIneq}
	\frac{\Gamma(x)^2}{\ch(\pi r)}\leq |\Gamma(x+ir)|^2\leq \Gamma(x)^2\quad \text{for }x\geq 0\text{ and } r\in \R. 
\end{equation}
Recall \eqref{widetildePhi} and \eqref{widehatPhi} that we define $\widehat{\phi}$ for $k\geq 0$. We also have
\begin{align}\label{widehatPhiinTypeWidetilde}
	\begin{split}
		&\widehat{\phi}(r)=\frac{\pi^2e^{\frac{1+k}2\pi i}}{\sh(\pi r)(\ch(2\pi r)+\cos(\pi k))\Gamma(\frac12-\frac k2+ir)\Gamma(\frac12-\frac k2-ir)}\\
		&\cdot\left\{\cos\tfrac{k\pi}2\ch(\pi r)\(\widetilde{\phi}(1+2ir)-\widetilde{\phi}(1-2ir)\)-i\sin\tfrac{k\pi}2\sh(\pi r)\(\widetilde{\phi}(1+2ir)+\widetilde{\phi}(1-2ir)\)\right\}. 
	\end{split}
\end{align} 
Like \cite[after (5.3)]{ADinvariants}, we define $\xi_k$ as
\begin{equation}\label{xi k r def}
	\xi_k(r)\defeq \frac{2i\pi^2 e^{\frac{1+k}2 \pi i}}{\Gamma(\frac12-\frac k2+ir)\Gamma(\frac12-\frac k2-ir)}.
\end{equation}
Then
\begin{equation}\label{xi k r bound}
	\xi_k(r)\left\{
	\begin{array}{ll}
		\ll 1&\quad \text{for }r\in [-1,1],\\
		\asymp |r|^ke^{\pi |r|}&\quad  \text{for }r\in(-\infty,-1]\cup[1,\infty).
	\end{array} 
	\right.
\end{equation}

We refer to \cite{Dunster1990a} for estimates on $J$-Bessel functions. Denote 
\[F_{\mu}(z)\defeq \frac{J_\mu(z)+J_{-\mu}(z)}{2\cos(\mu\pi/2)},\quad G_{\mu}(z)\defeq\frac{J_\mu(z)-J_{-\mu}(z)}{2\sin (\mu\pi/2)}.\]
As a result of the relationship $\overline{J_{2ir}(u)}=J_{-2ir}(u)$ for $r,u\in \R$ by \cite[(10.11.9)]{dlmf}, we have 
\[F_{2ir}(u)=\frac{\re J_{2ir}(u)}{\ch(\pi r)}\in \R,\quad G_{2ir}(u)=\frac{ \im J_{2ir}(u)}{\sh(\pi r)}\in \R. \]
Moreover, for $k\in \Z+\frac12$ and $k\geq 0$, 
\begin{equation}\label{widehatPhiinTypeFG}
	\widehat \phi(r)=\frac{\xi_k(r)\ch (\pi r)}{\ch(2\pi r)}\int_0^\infty \(G_{2ir}(u)\cos\frac{k\pi}2 -F_{2ir}(u)\sin\frac{k\pi}2\) \frac{\phi(u)}udu 
\end{equation}
and $\widehat{\phi}(r)=\widehat{\phi}(-r)$ for $r\in \R$ because $F_\mu(z)=F_{-\mu}(z)$ and $G_\mu(z)=G_{-\mu}(z)$.


\begin{lemma}\label{JBessel G F Bound}
	For $r\in[-1,1]$, uniformly and with absolute implied constants we have
	\begin{align}\label{RLessThan1}
		\begin{split}
			G_{2ir}(u)\ll \left\{\begin{array}{ll}
				\ln(\tfrac u2),&u\in[0,\frac 32],\vspace{4px}\\
				u^{-\frac32},&u\in[\frac 32,\infty). 
			\end{array}
			\right.
		\end{split}
	\end{align}
\end{lemma}
\begin{proof}
	
	First we deal with the range $u\in[0,\frac 32]$. The series expansion of $G_{ir}$ is given by \cite[(3.9), (3.16)]{Dunster1990a}:
	\begin{align*}
		G_{2ir}(u)&=\(\frac{4r\ch(\pi r)}{\pi\sh(\pi r)}\)^{\frac12}\sum_{\ell=0}^\infty \frac{(-1)^\ell (u^2/4)^\ell \sin(2r\ln(u/2)-\phi_{2r,\ell})}{\ell!\prod_{j=0}^{\ell^2}(j+4r^2)^{1/2}} \\
		&=\(\frac{\ch(\pi r)}{\pi r\sh(\pi r)}\)^{\frac12}\sin\(2r\ln\(\frac u2\)-\phi_{2r,0}\)+O\(\(\frac u2\)^2\). 
	\end{align*}
	where $\phi_{r,\ell}=\arg \Gamma(1+\ell+ir)$. The implied constant in the second equation is absolute. As a function of $r$, $\phi_{2r,0}\in C^\infty[0,1]$ and $\lim_{r\rightarrow 0}\phi_{2r,0}=0$. Then $\phi_{2r,0}/r=O(1)$ and
	\begin{align*}
		G_{2ir}(u)&\ll r^{-1}O\(\la 2r\ln\(\frac u2\)\ra+|\phi_{2r,0}|\)+O\(\(\frac u2\)^2\)\\
		&\ll \ln\(\frac u2\)+O(1). 
	\end{align*}
	
	For the range $u\geq \frac 32$, we check with \cite[(5.16)]{Dunster1990a} where $U_{s}(p)$ for $s\geq 0$ are fixed polynomials of $p$ whose lowest degree term is $p^{s}$:
	\begin{align*}
		G_{2ir}(u)&=\(\frac{4/\pi^2}{4r^2+u^2}\)^{\frac14}\(\frac{C}{\sqrt{4r^2+u^2}}+O\(\frac1{4r^2+u^2}\)\)\\
		&\ll \(r^2+\frac{u^2}4\)^{-\frac34}+O\(\(r^2+\frac{u^2}4\)^{-\frac54}\). 
	\end{align*}
	Our claimed bound is clear as $r^2\geq 0$. The implied constant above is absolute due to \cite[(3.3)]{Dunster1990a} and \cite[Chapter 8, \S13]{OlverAsymptSpecialFunctions1997} or by \cite[Chapter 10, (3.04)]{OlverAsymptSpecialFunctions1997}. 
	
\end{proof}

\begin{lemma}\label{RisReal 0 to 1}
	For $r\in [-1,1]$, we have
	\begin{equation}
		|\widetilde{\phi}(1+2ir)|\ll 1, \quad |\widehat{\phi}(r)|\ll_\ep (ax)^\ep. 
	\end{equation}
\end{lemma}
\begin{proof}
	A trivial bound of $J_{2ir}$ is given by the integral representation \cite[(10.9.4)]{dlmf}:
	\[J_\nu(z)=\frac{(z/2)^\nu}{\sqrt{\pi}\Gamma(\nu+\frac12)}\int_0^\pi \cos(z\cos \theta)(\sin\theta)^{2\nu}d\theta,\quad \re\nu>-\frac12.\] 
	Then we have $|J_{2ir}(u)|\leq \frac{\sqrt \pi}{|\Gamma(\frac12+2ir)|}$ and
	\[|\widetilde{\phi}(1+2ir)|\ll \int_{\frac{3a}{8x}}^{\frac{3a}{2x}}\frac{du}u\leq \ln 4 \quad \text{for }r\in [-1,1]. \]
	This implies that 
	\[\ch(\pi r)\int_0^\infty F_{2ir}(u)\frac{\phi(u)}u du\ll 1\quad \text{for }r\in [-1,1].\]
	Let the closed interval $[\alpha,\beta]=\emptyset$ when $\alpha>\beta$. With the help of \eqref{widehatPhiinTypeFG}, \eqref{xi k r bound} and Lemma~\ref{JBessel G F Bound} we get 
	\begin{align*}
		|\widehat{\phi}(r)|&\ll \frac{\ch(\pi r)}{\ch(2\pi r)}\(\left|\int_0^\infty G_{2ir}(u)\phi(u)\frac{du}u\right|+\left|\int_0^\infty F_{2ir}(u)\phi(u)\frac{du}u\right|\)\\		
		&\ll \int_{[\frac{3a}{8x},\frac 32]} \ln\(\frac u2\)\frac{du}u+\int_{[\frac 32,\frac{3a}{2x}]} u^{-\frac 52} du+O(1)\\
		&\ll \(\ln \frac{3a}{16x}\)^2+O(1)\ll (ax)^\ep. 
	\end{align*}
	The last inequality is because $a=4\pi\sqrt{\tilde m\tilde n}>0$ has a lower bound depending on $\nu$. 
\end{proof}

When we focus on the exceptional eigenvalues $\lambda_j\in [\frac3{16},\frac14)$ of $\Delta_k$, recall that $\lambda_j=\frac14+r_j^2$ for $r_j\in i(0,\frac14]$. By Proposition~\ref{specParaBoundWeightHalf}, if we write $t_j=\im r_j$, assuming $H_\theta$ \eqref{HTheta} we have an upper bound $t_j\leq \frac \theta 2$ when $r_j\neq \frac i4$. Moreover, since the exceptional eigenvalues are discrete, we also have a largest eigenvalue less than $\frac14$, hence a lower bound $\underline{t}>0$ (depending on $N$ and $\nu$) such that $t_j\geq\underline t$. 

\begin{lemma}\label{RisImBefore}
	With the hypothesis $H_\theta$ \eqref{HTheta} for $\theta\leq\frac16$, when $r=it$ and $t\in [\underline{t},\frac \theta2]$, we have
	\begin{equation}\label{RisIm}
		\widetilde{\phi}(1\pm 2t)\ll \(\frac ax\)^{\pm 2t}  \quad\text{and}\quad  \widehat{\phi}(r)\ll \(\frac ax\)^{2t}+\(\frac xa\)^{2t}\ll \(\frac ax\)^{\theta}+\(\frac xa\)^{\theta}.
	\end{equation}
Moreover, for $r=\frac i4$ we have
\begin{equation}\label{RisIm r=i/4}
	\widetilde{\phi}\(1\pm \frac 12\)\ll \(\frac ax\)^{\pm \frac12} \quad\text{and}\quad \widehat{\phi}\(\frac i4\)\ll \left\{\begin{array}{ll}
		(\frac xa)^{\frac12},&k=\frac12,\vspace{4px}\\
		(\frac ax)^{\frac12},&k=\frac 32. 
	\end{array}\right. 
\end{equation}
\end{lemma}
\begin{proof} 
	As in the previous lemma, when $t\in [\underline{t},\frac \theta2]$, the bound \cite[(10.9.4)]{dlmf} gives
	\[|J_{\pm 2t}(u)|\ll \frac{u^{\pm 2t}}{\Gamma(\frac12-\theta)}\quad\text{and}\quad |\widetilde{\phi}(1\pm 2t)|\ll \int_{\frac{3a}{8x}}^{\frac{3a}{2x	}}u^{\pm 2t}\;\frac{du}u\ll\(\frac ax\)^{\pm 	2t}.\]
	The bound for $\widehat{\phi}$ follows from \eqref{widehatPhiinTypeWidetilde}. 
	When $r=\frac i4$, by \cite[(10.16.1)]{dlmf} we have
	\[J_{-\frac12}(u)\ll u^{-\frac12}\quad \text{and}\quad J_{\frac12}(u)\ll u^{-\frac12}\sin u\leq u^{\frac12}. \]
	The bounds for $\widetilde{\phi}(1\pm\frac12)$ and $\widehat{\phi}(\frac i4)$ follow from the same process above with \eqref{widetildePhi} and \eqref{widehatPhi r=i/4}. 
\end{proof}

For the range $|r|\geq 1$ we have
\begin{lemma}\cite[Lemma~6.3]{ADinvariants}\label{phiBounds}
	Let $k=\frac12$ or $\frac 32$. Then
	\begin{equation}\label{Rlarge}
		\widehat{\phi}(r)\ll\left\{\begin{array}{ll}
			r^{k-\frac 32},&\ r\geq 1,\\
			r^k\min(r^{-\frac 32},r^{-\frac 52}\frac xT),&\ r\geq \max(\frac ax,1).
		\end{array}
		\right. 
	\end{equation}
\end{lemma} 
\begin{remark}
	In the original paper they stated the result for $k=\pm \frac12$. However, the power $r^k$ in the estimate above only arises from $\xi_k(r)e^{-\pi|r|}$ \eqref{xi k r def} and by \eqref{xi k r bound} we get the above lemma on weight $k=\frac 32$. 
\end{remark}

\subsection{A special test function}

Here we choose a special test function $\phi$ satisfying Setting~\ref{conditionphi} to compute the terms corresponding to the exceptional spectrum $r\in i(0,\frac14]$ in Theorem~\ref{mainThm}.

For $k=\frac12$ or $\frac 32$, let $\lambda\in [\frac3{16},\frac14)$ be an exceptional eigenvalue of $\Delta_k$ on $\Gamma_0(N)$, we set $\lambda=s(1-s)$ for $s\in(\frac12,\frac 34]$ and 
\[t=\im r=\sqrt{\tfrac14-\lambda}=\sqrt{\tfrac14-s(1-s)}=s-\tfrac12.\]
Since the exceptional spectrum is discrete, let the lower bound for $t>0$ be $\underline t$ depending on $N$ and $\nu$. Recall Setting~\ref{conditionATDelta}. Let $0<T'\leq T\leq \frac x3$ be $T'\defeq Tx^{-\delta}\asymp x^{1-2\delta}$. 

\begin{setting}\label{conditionPhiNew5.4}
	In addition to the requirement in Setting~\ref{conditionphi}, when $\frac a{x-T}\leq 1.999$, we pick $\phi$ as a smoothed function of this piecewise linear one 
	
	
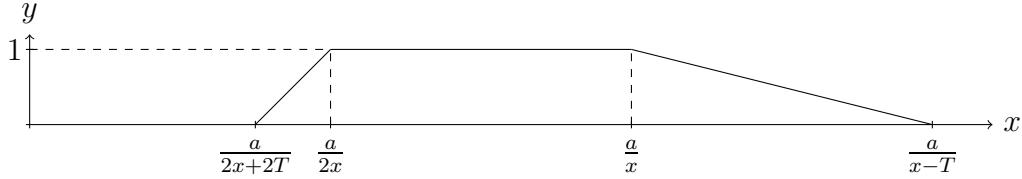
\begin{center}
		\begin{tikzpicture}
			\draw[->] (-0.05,0) -- (12.8, 0) node[right] {$x$};
			\draw[->] (0, -0.05) -- (0, 1.2) node[above] {$y$};
			\draw (3,0.05) -- (3,-0.05) node[below] {$\frac{a}{2x+2T}$};
			\draw (4,0.05) -- (4,-0.05) node[below] {$\frac{a}{2x}$};
			\draw (8,0.05) -- (8,-0.05) node[below] {$\frac{a}{x}$};
			\draw (12,0.05) -- (12,-0.05) node[below] {$\frac{a}{x-T}$};
			\draw[dashed] (0,1) -- (4,1);
			\draw (-0.05, 1) -- (0.05,1) node[left] {$1$};
			\draw[dashed] (4,0) -- (4,1);
			\draw[dashed] (8,0) -- (8,1);
			\draw[scale=2, domain=1.5:2, smooth, variable=\x, black] plot ({\x}, {\x-1.5});
			\draw[scale=2, domain=2:4, smooth, variable=\x, black] plot ({\x}, {0.5});
			\draw[scale=2, domain=4:6, smooth, variable=\x, black] plot ({\x}, {-0.25*\x+1.5});
		\end{tikzpicture}
	\end{center}
	where
	\begin{equation}\label{PhiDerivative}
		\left\{
		\begin{array}{ll}
			\phi'(u)=\tfrac{2x(x+T)}{aT}\quad  &u\in(\tfrac{a}{2x+2T-2T'},\;\tfrac{a}{2x+2T'}),\\
			\\
			\phi'(u)=-\tfrac{x(x-T)}{aT}\quad & u\in(\tfrac{a}{x-T'},\;\tfrac{a}{x-T+T'}), \\
			\\
			0\leq \phi'(u)\leq \tfrac{4x(x+T)}{aT}\quad & u\in (\tfrac{a}{2x+2T},\;\tfrac{a}{2x+2T-2T'})\cup (\tfrac{a}{2x+2T'},\;\tfrac{a}{2x}),\\
			\\
			0\geq \phi'(u)\geq -\tfrac{2x(x-T)}{aT}\quad & u\in (\tfrac{a}{x},\;\tfrac{a}{x-T'})\cup (\tfrac{a}{x-T+T'},\;\tfrac{a}{x-T}),\vspace{10px}\\
			\phi'(u)=0 \quad& \text{otherwise}.\tfrac{}{}
		\end{array}
		\right.
	\end{equation}
\end{setting}
The above choice for $\phi'$ is possible because there is no requirement for $\phi''(u)$ when $u\leq 2$ but for $u\rightarrow \infty$ in Setting \ref{conditionphi}.

Now we take $r=it\in i(0,\frac14]$. When $u\leq 1.999$, by the series expansion \cite[(10.2.2)]{dlmf}: 
\[J_\nu(z)=\(\frac z2\)^{\nu}\sum_{j=0}^\infty \frac{(-1)^j}{j!\Gamma(j+1+\nu)}\(\frac z2\)^{2j}, \]
we have
	\begin{equation}\label{J Bessel uniform Asymp u 1.999}
		J_{\pm 2t}(u)=\frac{(u/2)^{\pm 2t}}{\Gamma(1\pm 2t)}+O\(\(\frac u2\)^{2\pm 2t}\),\qquad 0<u\leq 1.999. 
	\end{equation} 
	The implied constant is absolute. Now we compute the bound for $\widetilde{\phi}$ and $\widehat{\phi}$.
	\begin{lemma}\label{R is Im precise bound}
		Assuming $H_\theta$ \eqref{HTheta} for $\theta\leq \frac16$ and with the choice of $\phi$ in Setting \ref{conditionPhiNew5.4}, when $r=it\in i(0,\frac14]$,
		\begin{align}
			\begin{split}
				\widetilde{\phi}(1-2t)&=\frac{1}{\Gamma(1-2t)}\int_{\frac{a}{2x}}^{\frac ax} \(\frac u2\)^{-2t}\frac{\phi(u)}udu + O\(a^{-2t}x^{2t-\delta}+1\)\\
				&= \frac{2^{2t}(2^{2t}-1)}{2t\Gamma(1-2t)}\(\frac xa\)^{2t}+O\(a^{-2t}x^{2t-\delta}+1\),
			\end{split}
		\end{align}
	\end{lemma}
	\begin{proof}
		When $1.999<\frac{a}{x-T}\leq \frac{3a}{2x}$, we get $x\ll a$ and $\widetilde{\phi}(1-2t)=O(1)$ by Lemma~\ref{RisImBefore}, so the lemma is true in this case. When $\frac{a}{x-T}\leq  1.999$, we have $a\ll x$ and with \eqref{J Bessel uniform Asymp u 1.999},
		\begin{align*}
			\widetilde{\phi}(1-2t)&=\int_0^\infty \frac{(u/2)^{-2t}}{\Gamma(1-2t)}\frac{\phi(u)}u du+O\(\int_0^\infty \(\frac u2\)^{2-2t}\frac{\phi(u)}u du\)\\
			&=\frac{2^{2t}}{\Gamma(1-2t)}\int_{\frac a{2x}}^{\frac ax} u^{-2t-1}du+\frac{2^{2t}}{\Gamma(1-2t)}\(\int_{\frac a{2x+2T}}^{\frac a{2x}}+\int_{\frac ax}^{\frac a{x-T}}\) u^{-2t-1}\phi(u)du\\
			&+O\(\int_0^\infty u^{1-2t}\phi(u) du\)\\
			&=:\frac{2^{2t}(2^{2t}-1)}{2t\Gamma(1-2t)}\(\frac xa\)^{2t}+(I_1+I_2)+O(I_3). 
		\end{align*} 
		Recall that we always have the lower bound $\underline t>0$ for $t=\im r$. A bound for $I_1$ and $I_2$ follows from the same process as \cite[Proof of Lemma~7.2]{QihangFirstAsympt}:
		\[I_1+I_2\ll \(\int_{\frac a{2x+2T}}^{\frac a{2x}}+\int_{\frac ax}^{\frac a{x-T}}\)u^{-2t-1}\phi(u)du\ll a^{-2t}x^{2t-\delta}. \]
		We also get
		\[I_3\ll \int_{\frac{3a}{8x}}^{\frac{3a}{2x}}u^{1-2t}du\ll\(\frac ax\)^{2-2t}\ll 1 \]
		and finish the proof. 
		
	\end{proof}
	
	\begin{lemma}\label{mainphiLemma}
		Assume $H_\theta$ \eqref{HTheta} for $\theta\leq \frac16$. For $r=it\in i(0,\frac \theta2]$ we have
		\[\widehat{\phi}(r)=\frac{e^{\frac{k\pi i}2}\cos(\pi t)\Gamma(\frac12+\frac k2+t)\Gamma(2t)}{\Gamma(\frac12-\frac k2+t)2^{2t}\pi^{2t}(\tilde m\tilde n)^{t}}\cdot\frac{(2^{2t}-1)x^{2t}}{2t} +O\(\frac{x^{2t-\delta}}{a^{2t}}+\frac{a^{2t}}{x^{2t}}+1\).\]
		Moreover, 
		\[\widehat{\phi}(\tfrac i4)=\left\{\begin{array}{lr}
			2e^{\frac{\pi i}4}(\sqrt2-1)(\frac xa)^{\frac12}+O(x^{-\delta}(\frac xa)^{\frac12}+1)\quad &\text{for\ }k=\frac12,\vspace{4px}\\
			e^{\frac{3\pi i}4}(1-\tfrac1{\sqrt 2})(\frac ax)^{\frac12}+O(x^{-\delta}(\frac ax)^{\frac12}+1)\quad &\text{for\ }k=\frac32. 
		\end{array}
		\right. \]
		The implied constants only depend on $N$ and $\nu$. 
	\end{lemma}
	\begin{proof}
		When $t\in [\underline t,\frac \theta2]$, we substitute Lemma~\ref{R is Im precise bound} into \eqref{widehatPhi} and use Lemma~\ref{RisImBefore} to get
		\begin{align*}
			\widehat{\phi}(it)&=\frac{i\pi^2 e^{\frac{k\pi i}2}\(\cos(\frac{k\pi }2-\pi t)\widetilde{\phi}(1-2t)-\cos(\frac{k\pi }2+\pi t)\widetilde{\phi}(1+2t)\)}{i\sin(\pi t)\cos(2\pi t)\Gamma(\frac12-\frac k2-t)\Gamma(\frac12-\frac k2+t)}\\
			&=\frac{\pi^2 e^{\frac{k\pi i}2} \cos(\frac {k\pi}2-\pi t)2^{2t}(2^{2t}-1)( x/a)^{2t}}{\sin(\pi t)\cos(2\pi t)\Gamma(\frac12-\frac k2-t)\Gamma(\frac12-\frac k2+t)2t\Gamma(1-2t)} +O\(\frac{x^{2t-\delta}}{a^{2t}}+\frac{a^{2t}}{x^{2t}}+1\). 
		\end{align*}
		With the help of the functional equation of the $\Gamma$ function \[\Gamma(z)\Gamma(1-z)=\frac{\pi}{\sin(\pi z)}\quad \text{for\ }z\in \C\setminus \Z\]
		and the trigonometric identities
		\[\sin(\tfrac \pi 2-x)=\cos x,\quad 2\cos x\cos y=\cos(x+y)+\cos(x-y)\quad \text{for\ }x,y\in \R,\]
		we have
		\begin{align*}
			&\frac{\pi}{\sin(\pi t)\Gamma(1-2t)}=2\cos(\pi t)\Gamma(2t),\\
			&\frac{\pi}{\Gamma(\frac12-\frac k2-t)}=\Gamma(\tfrac12+\tfrac k2+t)\cos(\tfrac {k\pi}2+\pi t),\\
		\text{and}\quad 	&2\cos(\tfrac {k\pi}2-\pi t)\cos(\tfrac {k\pi}2+\pi t)=\cos(2\pi t). 
		\end{align*} 
Then the first part of the lemma follows. 
		The implied constant only depends on $N$ and $\nu$ because $t\in [\underline{t},\frac \theta 2]$ is bounded above and below away from 0.

		When $t=\frac14$, the process is similar to the proof of Lemma~\ref{R is Im precise bound} with the help of \eqref{widehatPhi r=i/4}. First we deal with the case $k=\frac12$ with $\cos u=1+O(u^2)$ for $u\in [0,\frac \pi 2]$. Thus, when $\frac{a}{x-T}>\frac \pi 2$, we have $x\ll a$ and $\widehat\phi(\frac i4)=O(1)$ in this case. When $\frac{a}{x-T}\leq \frac \pi 2$, we have $a\ll x$ and
		\begin{align*}
			\widehat{\phi}(\tfrac i4)&=e^{\frac{\pi i}4}\int_0^\infty \phi(u) u^{-\frac32}du+O\(\int_0^\infty \phi(u)u^{\frac 12}du\)\\
			&=e^{\frac{\pi i}4}\int_{\frac{a}{2x}}^{\frac ax}u^{-\frac32}du+e^{\frac{\pi i}4}\(\int_{\frac a{2x+2T}}^{\frac a{2x}}+\int_{\frac ax}^{\frac a{x-T}}\) u^{-\frac32}\phi(u)du+O(1)\\
			&=e^{\frac{\pi i}4}(2\sqrt 2-2)\(\frac xa\)^{\frac12}+O\(x^{-\delta}\(\frac xa\)^{\frac12}\)+O(1).
		\end{align*}
	The case for $k=\frac32$ is similar using $\sin u=u+O(u^3)$ for $u\in[0,\frac \pi 2]$. 
	\end{proof}

	\section{Proof of Theorem~\ref{mainThm} and Theorem~\ref{mainThmLastSec}}
	The proof depends on the following two propositions for the Fourier coefficients of Maass forms, which were originally obtained for the discrete spectrum in \cite[Theorem~4.1]{ADasymptotic} and \cite[Theorem~4.3]{ahlgrendunn}. The author proved the generalized propositions in \cite[\S 8]{QihangFirstAsympt} to include the continuous spectrum. 
	Recall our notations in Settings \ref{conditionATDelta} and \ref{conditionphi}.
	Suppose that for some $\beta\in (\frac12,1)$,
	\begin{equation}\label{AbsSumSnn}
		\sum_{N|c>0}\frac{|S(n,n,c,\nu)|}{c^{1+\beta}}\ll_{N,\nu,\ep} | \tilde{n}|^\ep
	\end{equation}
	Then we have the following proposition: 
	\begin{proposition}[{\cite[Proposition~8.1]{QihangFirstAsympt}}]\label{AdsDukeEsstBd}
		Suppose that $\nu$ is a multiplier on $\Gamma=\Gamma_0(N)$ of weight $k=\pm\frac12$ which satisfies \eqref{AbsSumSnn} for some $\beta\in (\frac12,1)$. Let $\rho_j(n)$ denote the Fourier coefficients of an orthonormal basis $\{v_j(\cdot)\}$ of $\LEigenform_{k}(N,\nu)$. For each singular cusp $\mathfrak{a}$ of $(\Gamma,\nu)$, let $\Ea(\cdot,s)$ be the associated Eisenstein series. Let $\rho_{\mathfrak{a}}(n,r)$ be defined as in \eqref{FourierExpEsst}. Then for $x>0$ we have 
		\begin{align*}
			x^{k \sgn  \tilde{n}}| \tilde{n}|\(\sum_{x\leq r_j\leq 2x}|\rho_j(n)|^2e^{-\pi r_j}+\right.  \sum_{\mathrm{singular\ }\mathfrak{a}} &\left. \int_{|r|\in[x, 2x]} |\rho_{\mathfrak{a}}(n,r)|^2 e^{-\pi |r|}dr\) \\
			& \ll_{N,\nu,\ep}x^2+| \tilde{n}|^{\beta+\ep}x^{1-2\beta}\log^\beta x.
		\end{align*} 
	\end{proposition} 
	\begin{remark}
		In Definition~\ref{Admissibility}, an admissible multiplier satisfies \eqref{AbsSumSnn} with $\beta=\frac12+\ep$ for any $\ep$. 
	\end{remark}

	\begin{proposition}[{\cite[Proposition~8.2]{QihangFirstAsympt}}]\label{AhgDunEsstBd}
		Suppose that $\nu$ is a weight $k=\pm\frac12$ admissible multiplier on $\Gamma=\Gamma_0(N)$ with $M$, $D$ and $B$ given in Definition~\ref{Admissibility}. Let $\rho_j(n)$ denote the Fourier coefficients of an orthonormal basis $\{v_j(\cdot)\}$ of $\LEigenform_{k}(N,\nu)$. For each singular cusp $\mathfrak{a}$ of $(\Gamma,\nu)$, let $\Ea(\cdot,s)$ be the associated Eisenstein series. Let $\rho_{\mathfrak{a}}(n,r)$ be defined as in \eqref{FourierExpEsst}. Suppose $x\geq 1$. 
		For $ n\neq 0$ we factor $B\tilde n=t_nu_n^2w_n^2$ where $t_n$ is square-free, $u_n|M^\infty$ is positive and ${(w_n,M)=1}$. Then we have
		\[x^{k \sgn  {\tilde n}}|{\tilde n}|\(\sum_{|r_j|\leq x} \frac{|\rho_j(n)|^2}{\ch \pi r_j} +\sum_{\mathrm{singular\ }\mathfrak{a}}\int_{-x}^{x} \frac{|\rho_{\mathfrak{a}}(n,r)|^2}{\ch \pi r}dr\)  \ll_{N,\nu,\ep} \( |{\tilde n}|^{\frac{131}{294}}+u_n\) x^{3}|\tilde {n}|^\ep. \]
	\end{proposition} 
	
	\begin{remark}
		We make some remarks about the weight $k$:  
		\begin{itemize}
			\item The trace formula (Theorem~\ref{traceFormula Theorem}) works for $k=\frac12$ and $\frac 32$.
			\item The estimates on $\widehat{\phi}$ and $\widetilde{\phi}$ in the previous section work for $k=\frac12$ and $\frac 32$.
			\item The above two propositions work for $k=\frac12$ and ${-\tfrac 12}$. 
		\end{itemize}
	Therefore, in this section, we separate the proof of Theorem~\ref{mainThm} into two cases $k=\frac12$ and ${-\tfrac 12}$. In the second case we will apply the Maass lowering operator $L_{\frac 32}$ \eqref{Maass Lowering op} to connect the eigenforms of weight $\frac 32$ and weight $-\frac12$. 
	\end{remark}

We declare that all implicit constants for the bounds in this section depend on $N$, $\nu$ and $\ep$, and we drop the subscripts unless specified.

	Since the exceptional spectral parameter $r_j\in i(0,\frac14]$ of Laplacian $\Delta_k$ on $\Gamma=\Gamma_0(N)$ is discrete, $t_j=\im r_j$ has a positive lower bound denoted as $\underline t>0$ depending on $N$ and $\nu$. We also have $2 \im r_\Delta\leq \theta$ assuming $H_\theta$ \eqref{HTheta} by Theorem~\ref{specParaBoundWeightHalf}. For simplicity let 
	\[A(m,n)\defeq(\tilde{m}^{\frac{131}{294}}+u_m)^{\frac12}(\tilde{n}^{\frac{131}{294}}+u_n)^{\frac12}\ll( \tilde{m}  \tilde{n})^{\frac{131}{588}}
	+  \tilde{m}^{\frac{131}{588}} u_n^{\frac12}
	+ \tilde{n}^{\frac{131}{588}} u_m^{\frac12}+
	(u_mu_n)^{\frac12}, \]
	then
	\begin{align}\label{AuDef}
		\begin{split}
			A_u(m,n)&\defeq A(m,n)^{\frac14}  (\tilde{m}  \tilde{n})^{\frac3{16}}\\
			&\ll ( \tilde{m}  \tilde{n})^{\frac{143}{588}}+\tilde{m}^{\frac{143}{588}}\tilde{n}^{\frac 3{16}}\,u_n^{\frac18}+\tilde{m}^{\frac 3{16}}u_m^{\frac18}\tilde{n}^{\frac{143}{588}}+ (\tilde{m}  \tilde{n})^{\frac3{16}}(u_mu_n)^{\frac18}.
		\end{split}
	\end{align}
Recall the notations in Setting~\ref{conditionATDelta} and Setting~\ref{conditionphi}. The following inequalities will be used later in the proof:
\begin{equation}
	\label{A leq Au leq mn 1/4} 
	A(m,n)\ll A_u(m,n)\ll (\tilde m\tilde n)^{\frac14};
\end{equation}
\begin{equation}
	\label{Pow ax A leq Au when x large}
	 \(\frac ax\)^{\beta}A(m,n)\ll A_u(m,n) \quad \text{for }0\leq \beta\leq \frac 32,\quad \text{when }x\gg A_u(m,n)^2. 
\end{equation}
	
	\subsection{On the case $k=\frac12$}
	Let $\rho_j(n)$ denote the coefficients of an orthonormal basis $\{v_j(\cdot)\}$ of $\LEigenform_{\frac12}(N,\nu)$. For each singular cusp $\mathfrak{a}$ of $\Gamma=\Gamma_0(N)$, let $\rho_{\mathfrak{a}}(n,r)$ be defined as in \eqref{FourierExpEsst}. Recall the definition of $\tau_j(m,n)$ in Theorem~\ref{mainThm} and the notations in Settings \ref{conditionATDelta} and \ref{conditionphi}. We claim the following proposition: 
	\begin{proposition}\label{PrereqPropGEN}
		With the same setting as Theorem~\ref{mainThm} for $k=\frac12$, when $2x\geq A_u(m,n)^2$, we have 
		\begin{align}\label{PrereqPropGENequation}
			\begin{split}
				\sum_{\substack{x< c\leq 2x\\N|c}} \frac{S(m,n,c,\nu)}{c} -\sum_{r_j\in i (0,\frac14]}(2^{2s_j-1}-1)\tau_j(m,n)\frac{x^{2s_j-1}}{2s_j-1}\ll \(x^{\frac 16}+A_u(m,n)\)(\tilde{m}\tilde{n}x)^\ep.
			\end{split}
		\end{align}
	\end{proposition}
	
	We first show that Proposition~\ref{PrereqPropGEN} implies Theorem~\ref{mainThm} in the case $k=\frac12$, which follows from a similar process as \cite[after Proposition~9.1]{QihangFirstAsympt}. Recall that $2\im r_j=2s_j-1$ for $r_j\in i(0,\frac14]$ and that the corresponding exceptional eigenvalue $\lambda_j=\frac14+r_j^2=s_j(1-s_j)$. The sum to be estimated is 
	\begin{equation}\label{ToEstmtMainThm}
		\sum_{N|c\leq X}\frac{S(m,n,c,\nu)}c-\sum_{r_j\in i(0,\frac14]}\tau_j(m,n)\frac{X^{2\im r_j}}{2\im r_j} ,
	\end{equation}
	where 
	\[\tau_j(m,n)=2i^{\frac12}\overline{\rho_j(m)}\rho_j(n)\pi^{1-2s_j}(4\tilde m \tilde n)^{1-s_j}\frac{\Gamma(s_j+\frac 14)\Gamma(2s_j-1)}{\Gamma(s_j-\frac 14)}.\]
	Since $t_j=\im r_j\in[\underline{t},\frac14]$ and $s_j=\im r_j+\frac12\in [\underline t+\frac12,\frac 34]$, 
	the quantity 
	\[\pi^{1-2s_j}4^{1-s_j}\frac{\Gamma(s_j+\frac14)\Gamma(2s_j-1)}{\Gamma(s_j-\frac14)}\]
	is bounded from above and below. By Proposition~\ref{AhgDunEsstBd}, 
	\begin{equation}\label{Bound for tau_j}
		\frac{\tau_j(m,n)}{2s_j-1}\ll |\rho_j(m)\rho_j(n)|(\tilde m\tilde n)^{1-s_j}\ll A(m,n)(\tilde m\tilde n)^{\frac12-s_j+\ep}. 
	\end{equation}
	When $X\ll A_u(m,n)^2$, since $A(m,n)\ll (\tilde{m}\tilde{n})^{\frac14}$ by \eqref{A leq Au leq mn 1/4}, 
	\begin{align}\label{taujestGEN}
		\begin{split}
			\tau_j(m,n)\frac{X^{2s_j-1}}{2s_j-1}&\ll  A(m,n)|\tilde{m}\tilde{n}|^{\frac12-s_j+\ep}A_u(m,n)^{4s_j-2}\\
			&= A(m,n)^{s_j+\frac12}|\tilde{m}\tilde{n}|^{\frac18-\frac14s_j+\ep}\ll A_u(m,n)(\tilde m\tilde n)^\ep  . 
		\end{split}
	\end{align}
	So in this case we get Theorem~\ref{mainThm} where the $\tau_j$ terms are absorbed in the errors. 
	
	When $X\geq A_u(m,n)^2$, the segment for summing Kloosterman sums on $1\leq c\leq A_u(m,n)^2$ contributes a $O_{\nu,\ep}(A_u(m,n)|\tilde{m}\tilde{n}|^{\ep})$ by condition (2) of Definition \ref{Admissibility}. The segment for $A_u(m,n)^2\leq c\leq X$ can be broken into no more than $O(\log X)$ dyadic intervals $x<c\leq 2x$ with $A_u(m,n)^2\leq x\leq \frac X2$ and we use Proposition~\ref{PrereqPropGEN} for both the Kloosterman sum and the $\tau_j$ terms. In summing dyadic intervals, for each $r_j\in i(0,\frac14]$, we get
	\begin{align*}
		\sum_{\ell=1}^{\ceil{\log_2 \(X/A_u(m,n)^2\)}}&\frac{(2^{2s_j-1}-1)\tau_j(m,n)}{2s_j-1}
		\(\frac{X}{2^\ell}\)^{2s_j-1}\\
		&\qquad =\frac{\tau_j(m,n)}{2s_j-1}X^{2s_j-1}
		\( 1 - 2^{(1-2s_j)\ceil{\log_2 \(X/A_u(m,n)^2\)}}\).
	\end{align*}
	The difference between the above quantity and the quantity $\tau_j(m,n)\dfrac{X^{2s_j-1}}{2s_j-1}$ in \eqref{ToEstmtMainThm} is
	\begin{equation}\label{differenceBetweenEstmtedTauAndTrueTauGEN}
		{\tau_j(m,n)}\frac{X^{2s_j-1}}{2s_j-1}\cdot 2^{(1-2s_j)\ceil{\log_2 \(X/A_u(m,n)^2\)}}\ll \frac{\tau_j(m,n)}{2s_j-1}A_u(m,n)^{4s_j-2}\ll A_u(m,n)
	\end{equation}
	by \eqref{Bound for tau_j}. In conclusion, for $X\geq A_u(m,n)^2$ we get
	\begin{align*}
		&\sum_{\substack{N|c\leq X}}\frac{S(m,n,c,\nu)}{c} -\sum_{r_j\in i (0,\frac14]}\tau_j(m,n)\frac{X^{2s_j-1}}{2s_j-1}\\
		&=\sum_{\substack{A_u(m,n)^2<c\leq X}} \frac{S(m,n,c,\nu)}{c}-\sum_{r_j\in i (0,\frac14]}\tau_j(m,n)\frac{X^{2s_j-1}}{2s_j-1}+O(A_u(m,n)|\tilde{m}\tilde{n}|^{\ep})\\
		&=\sum_{\ell=1}^{\ceil{\log_2 \(X/A_u(m,n)^2\)}}\(\sum_{\substack{\frac{X}{2^\ell}<c\leq \frac{X}{2^{\ell-1}}}} \frac{S(m,n,c,\nu)}{c} -\sum_{r_j\in i (0,\frac14]}\frac{(2^{2s_j-1}-1)\tau_j(m,n)}{2s_j-1}\(\frac{X}{2^\ell}\)^{2s_j-1}\)\\
		&\quad + O(A_u(m,n)|\tilde{m}\tilde{n}|^{\ep})\\
		&\ll \(X^{\frac 16}+A_u(m,n)\)|\tilde{m}\tilde{n}X|^\ep
	\end{align*}
	where the second equality follows from \eqref{differenceBetweenEstmtedTauAndTrueTauGEN} and the last inequality is by Proposition~\ref{PrereqPropGEN}. Theorem~\ref{mainThm} follows in the case $k=\frac12$.

	The proof of Proposition~\ref{PrereqPropGEN} takes the rest of this subsection. 
	For $r_j\in i(0,\frac14]$, by Proposition~\ref{AhgDunEsstBd} we have
	\[\sqrt{\tilde{m}\tilde{n}}\; \overline{\rho_j(m)}\rho_j(n)\ll A(m,n)(\tilde{m}\tilde{n})^{\ep}. \]
	Recall that $a=4\pi\sqrt{\tilde m\tilde n}$ and $\delta=\frac13$ in Setting~\ref{conditionATDelta}. Thanks to $H_{\frac 7{64}}$ \eqref{HTheta} and Proposition~\ref{specParaBoundWeightHalf}, when $r_j=it_j\in i(0,\frac \theta2]$ we have $2t_j<\delta=\frac 13$. Since $2x\geq A_u(m,n)^2$ by hypothesis, it follows from \eqref{Pow ax A leq Au when x large} that 
	\begin{align*}
	\sqrt{\tilde{m}\tilde{n}}\; \overline{\rho_j(m)}\rho_j(n)\(\frac{x^{2t_j-\delta}}{a^{2t_j}}+\frac {a^{2t_j}}{x^{2t_j}}+1\)\ll A_u(m,n)(\tilde m\tilde n)^\ep.
	\end{align*}
	Applying Lemma~\ref{mainphiLemma} where $t_j\in [\underline t,\frac\theta 2]$ and recalling the definition of $\tau_j$ in Theorem~\ref{mainThm}, we get
	\begin{align}\label{taujReference5.4GEN}
		\begin{split}
			4\sqrt{\tilde{m}\tilde{n}}&\;\frac{\overline{\rho_j(m)}\rho_j(n)}{\ch \pi r_j}\widehat{\phi}(r_j)\\
			&=(2^{2s_j-1}-1)\tau_j(m,n)\frac{x^{2s_j-1}}{2s_j-1}+O\(A_u(m,n)(\tilde{m}\tilde{n})^\ep\). 
	\end{split}\end{align}
	When $r_j=\frac i4$ and $k=\frac12$, Lemma~\ref{mainphiLemma} and \eqref{Pow ax A leq Au when x large} give
		\begin{align}\label{taujReference5.4GEN r=i/4}
		\begin{split}
			4\sqrt{\tilde{m}\tilde{n}}\;\frac{\overline{\rho_j(m)}\rho_j(n)}{\cos\frac \pi 4}\widehat{\phi}(\tfrac i4)=2(\sqrt 2-1)\tau_j(m,n)x^{\frac12}+O\(x^{\frac12-\delta}(\tilde m\tilde n)^{\ep}\) .
	\end{split}\end{align}

	With the help of \eqref{taujReference5.4GEN} and \eqref{taujReference5.4GEN r=i/4} we break up the left hand side of \eqref{PrereqPropGENequation} to obtain the following analogue to \cite[(9.8)]{QihangFirstAsympt}: 
	\begin{align}\label{mainDiffrenceGEN}
		\begin{split}
			&\left|\sum_{\substack{x<c\leq 2x\\N|c}} \frac{S(m,n,c,\nu)}{c} -\sum_{r_j\in i (0,\frac14]}(2^{2s_j-1}-1)\tau_j(m,n)\frac{x^{2s_j-1}}{2s_j-1} \right|\\
			\leq& \left|\sum_{\substack{x<c\leq 2x\\N|c}} \frac{S(m,n,c,\nu)}{c} -\sum_{N|c>0} \frac{S(m,n,c,\nu)}{c}\phi\(\frac ac\)\right| +O\(\(x^{\frac12-\delta}+A_u(m,n)\)(\tilde{m}\tilde{n})^{\ep}\)\\
			&+ \left|\sum_{N|c>0} \frac{S(m,n,c,\nu)}{c}\phi\(\frac ac\)-4\sqrt{\tilde{m}\tilde{n}}\sum_{r_j\in i (0,\frac14]}\frac{\overline{\rho_j(m)}\rho_j(n)}{\ch \pi r_j}\widehat{\phi}(r_j)\right|\\
			=:&\;S_1+O\(\(x^{\frac12-\delta}+A_u(m,n)\)(\tilde{m}\tilde{n})^{\ep}\)+S_2. 
		\end{split}
	\end{align}
	The first sum $S_1$ above can be estimated by condition (2) of Definition \ref{Admissibility} as
	\begin{align}\label{TraceSmoothingGEN}
		\begin{split}
			S_1\leq\sum_{\substack{x-T\leq c\leq x\\2x\leq c\leq 2x+2T\\N|c}}\frac{|S(m,n,c,\nu)|}{c}\ll_{N,\nu,\delta,\ep} x^{\frac12-\delta}(\tilde{m}  \tilde{n}x)^\ep.
		\end{split}
	\end{align}
	We then prove a bound for $S_2$. By Theorem~\ref{traceFormula Theorem}, we have
	\begin{equation}\label{S2 weight 12}
		S_2\ll |\mathcal U_{\frac12}|+\la \sqrt{\tilde m\tilde n}\sum_{r_j\geq 0}\frac{\overline{\rho_j(m)}\rho_j(n)}{\ch \pi r_j} \widehat{\phi}(r_j)+\sqrt{\tilde m\tilde n}\sum_{\mathrm{singular\ }\mathfrak{a}}\int_{-\infty}^\infty \overline{\rho_\mathfrak{a}(m,r)}\rho_\mathfrak{a}(n,r)\frac{\widehat{\phi}(r)}{\ch \pi r }dr\ra. 
	\end{equation}

	\subsubsection{Contribution from holomorphic forms}
	For $k=\frac 12$ or $\frac 32$, recall the notation $\mathscr{B}_k$ before Theorem~\ref{traceFormula Theorem}. For $l\geq 1$, let $\{F_{j,l}(\cdot)\}_j$ be an orthonormal basis of $S_{k+2l}(N,\nu)$ with Fourier coefficient $a_{F,j,l}$. 
	By Proposition~\ref{Level lifting: holo cusp forms, final bound}, uniformly for every $l\geq 1$ with $d_l\defeq\dim S_{k+2l}(N,\nu)$, we have $k+2l\geq \frac 52$ and
	\begin{align*}
		&\frac{\Gamma(k+2l-1)}{(4\pi)^{k+2l-1}(\tilde m\tilde n)^{\frac{k+2l-1}2}}\sum_{j=1}^{d_{l}} \overline{a_{F,j,l}(m)}a_{F,j,l}(n)\\
		&\leq \(\frac{\Gamma(k+2l-1)}{(4\pi \tilde{n})^{k+2l-1}}\sum_{j=1}^{d_{l}}|a_{F,j,l}(m)|^2\)^{\frac12}\(\frac{\Gamma(k+2l-1)}{(4\pi \tilde{m})^{k+2l-1}}\sum_{j=1}^{d_{l}}|a_{F,j,l}(n)|^2\)^{\frac12}\\
		&\ll (\tilde m^{\frac{19}{42}}+u_m)^{\frac12}(\tilde n^{\frac{19}{42}}+u_n)^{\frac12}(\tilde m\tilde n)^\ep. 
	\end{align*}
We also have
\begin{equation*}
\sum_{l=1}^\infty (k+2l -1)\,|\widetilde{\phi}(k+2l)|\ll 1+\frac ax 
\end{equation*}
by \cite[Lemma~5.1 and proof of~Lemma 7.1]{dunn} and Lemma~\ref{RisImBefore}. Note that \cite[Lemma~5.1]{dunn} is only for $k=\frac12$, while the same process works for $k=\frac 32$. 
Then the contribution from $\mathcal U_k$ is
	\begin{align*}
		\mathcal{U}_k&= \sum_{l=1}^\infty \frac{k+2l-1}{4\pi}\,\widetilde{\phi}\,(k+2l)\frac{\Gamma(k+2l -1)}{(4\pi)^{k+2l-1}(\tilde m\tilde n)^{\frac{k+2l-1}2}}\sum_{j=1}^{d_l} \overline{a_{F,j,l}(m)}a_{F,j,l}(n)\\
		&\ll\(1+\frac ax\) (\tilde m^{\frac{19}{42}}+u_m)^{\frac12}(\tilde n^{\frac{19}{42}}+u_n)^{\frac12}(\tilde m\tilde n)^\ep. 
	\end{align*}
Recall $a=4\pi\sqrt{\tilde m\tilde n}$ and \eqref{AuDef} for the definition of $A_u(m,n)$. 
We can directly calculate
\[(\tilde m^{\frac{19}{42}}+u_m)^{\frac12}(\tilde n^{\frac{19}{42}}+u_n)^{\frac12}\ll A_u(m,n).  \]
By the hypothesis $2x\geq A_u(m,n)^2$, we also get
\[(\tilde m^{\frac{19}{42}}+u_m)^{\frac12}(\tilde n^{\frac{19}{42}}+u_n)^{\frac12}\cdot\frac ax\ll(\tilde m^{\frac{19}{42}}+u_m)^{\frac12}(\tilde n^{\frac{19}{42}}+u_n)^{\frac12}\cdot\frac a{A_u(m,n)^2}\ll A_u(m,n). \] 
Finally we conclude
	\begin{equation}\label{Trace estmt Holomorphic}
		\mathcal{U}_k\ll A_u(m,n)(\tilde m\tilde n)^\ep\qquad \text{for }k=\tfrac12\text{ or }\tfrac32.  
	\end{equation}


	\subsubsection{Contribution from Maass cusp forms and Eisenstein series. }\label{5.1.2}
	We combine the two propositions at the beginning of this section and bounds on $\widehat{\phi}$ in Section 4 to estimate the contribution from the remaining part of $S_2$ \eqref{S2 weight 12} other than $\mathcal U_k$. The process is the same as \cite[\S 9.1]{QihangFirstAsympt} for $|r|\leq 1$ as $\widehat{\phi}$ shares the same bound as $\check{\phi}$ there. We record the bounds in the following equations. 
	
	Fix $k=\frac12$. In the following estimations we focus on the discrete spectrum $r_j\geq 0$ because each bound for $r_j\in [a,b]$ for any interval $[a,b]\subset \R$ is the same as the bound for $r\in [a,b]\cup [-b,-a]$ in the continuous spectrum. This is a direct result from Proposition~\ref{AdsDukeEsstBd} and Proposition~\ref{AhgDunEsstBd}. Recall that $2x\geq A_u(m,n)^2$ in the assumption of Proposition~\ref{PrereqPropGEN}.

	For $r\in [0,1)$, we apply Lemma~\ref{RisReal 0 to 1}, Proposition~\ref{AhgDunEsstBd} and Cauchy-Schwarz to get
	\begin{equation}\label{GeneralTraceRisSmall}
		\sqrt{ \tilde{m}  \tilde{n}}\sum_{r\in [0,1)} \left|\frac{\overline{\rho_j(m)}\rho_j(n)}{\ch \pi r_j} \widehat{\phi}(r_j)\right| 
		\ll  A(m,n) ( \tilde{m}  \tilde{n}x)^\ep.
	\end{equation}
	
	For $r\in[1,\frac ax)$, we apply Proposition~\ref{AhgDunEsstBd} and $\widehat{\phi}(r)\ll r^{-1}$ from \eqref{Rlarge}.  
	Since
	\begin{align}
		\begin{split}
			S(R)\defeq \sqrt{ \tilde{m}  \tilde{n}}\sum_{r\in [1,R]} \left|\frac{\overline{\rho_j(m)}\rho_j(n)}{\ch \pi r_j} \right|
			\ll  A(m,n)R^{\frac 52}(\tilde{m} \tilde{n})^{\ep},
		\end{split}
	\end{align}
with the help of \eqref{Pow ax A leq Au when x large} we have
	\begin{align}\label{GeneralTraceRlarge1}
		\begin{split}
			\sqrt{ \tilde{m}  \tilde{n}}\sum_{r\in [1,\frac a{x})} \left|\frac{\overline{\rho_j(m)}\rho_j(n)}{\ch \pi r_j} \widehat{\phi}(r_j)\right|&\ll r^{-1}S(r)\Big|_{r=1}^{\frac a{x}}+\int_1^{\frac{a}{x}}S(r) r^{-2}dr\\
			&\ll A(m,n) \(\frac ax\)^{\frac 32}(\tilde{m}  \tilde{n}x)^\ep \ll A_u(m,n)(\tilde{m}  \tilde{n}x)^\ep . 
		\end{split}
	\end{align}
	
	Let 
	\[P(m,n)\defeq 2(\tilde{m}\tilde{n})^{\frac18}A(m,n)^{-\frac12}\geq 1.\] 
	Divide $r\geq \max(\frac ax,1)$ into two parts: 
	$\max\(\frac ax,1\)\leq r< P(m,n)$
	and 
	$r \geq  \max\(\frac ax,1,P(m,n)\)$. 
	We apply Proposition~\ref{AhgDunEsstBd} on the first range and $\widehat{\phi}(r)\ll r^{-1}$ from \eqref{Rlarge} to get
	\begin{equation}\label{TraceRlargeGeneralFirstPartEstmt}
		\sqrt{ \tilde{m} \tilde{n}} \sum_{\max(\frac ax,1)\leq r_j<P(m,n)} \left| \frac{\overline{\rho_j(m)} \rho_j(n)}{\ch \pi r_j}\widehat{\phi}(r_j)\right|\\
		\ll A_u(m,n)(  \tilde{m}  \tilde{n}x)^\ep
	\end{equation}
	by partial summation as in \eqref{GeneralTraceRlarge1}. We divide the second range into dyadic intervals $C\leq r_j< 2C$. Applying Proposition~\ref{AdsDukeEsstBd} with $\beta=\frac12+\ep$ and $\widehat{\phi}(r)\ll \min(r^{-1},r^{-2}\frac xT)$ from \eqref{Rlarge}, we get  
	\begin{align}\label{TraceRlargeGeneralSecondEstmtDyadic}
		\begin{split}
			\sqrt{ \tilde{m}\tilde{n}}& \sum_{C\leq r_j< 2C} \left| \frac{\overline{\rho_j(m)} \rho_j(n)}{\ch \pi r_j}\widehat{\phi}(r_j)\right|\\
			&\ll \min\(C^{-1},C^{-2}\frac xT\)C^{-\frac12}\(C^2+ (\tilde{m}^{\frac14} +\tilde{n}^{\frac14})C+ (\tilde{m}  \tilde{n})^{\frac14}\)( \tilde{m}  \tilde{n}x)^\ep\\
			&\ll \(\min\(C^{\frac12},C^{-\frac 12}\frac xT\)+ (\tilde{m}^{\frac14} +\tilde{n}^{\frac14})C^{-\frac12}+ (\tilde{m}  \tilde{n})^{\frac14}C^{-\frac32}\)( \tilde{m}  \tilde{n}x)^\ep.
		\end{split}
	\end{align}
	Next we sum over dyadic intervals. For the first term $\min(C^{\frac12},C^{-\frac 12}\frac xT)$, when
	\[\min\(C^{\frac12},C^{-\frac 12}\frac xT\)=C^{\frac12}:\quad \sum_{\substack{j\geq 1:\ 2^jC=\frac xT\\ C\geq P(m,n)}} C^{\frac12}\leq \sum_{j=1}^\infty 2^{-\frac j2}\(\frac xT\)^{\frac12}\ll \(\frac xT\)^{\frac12}, \]
	and when
	\[\min\(C^{\frac12},C^{-\frac 12}\frac xT\)=C^{-\frac12}\frac xT:\quad \sum_{j\geq 0:\ C=2^{j}\frac xT} C^{-\frac12}\frac xT\leq \sum_{j=0}^\infty 2^{-\frac j2}\(\frac xT\)^{\frac12}\ll \(\frac xT\)^{\frac12} . \]
	So after summing up from \eqref{TraceRlargeGeneralSecondEstmtDyadic}, recalling $T\asymp x^{1-\delta}$ in Setting \ref{conditionATDelta}, using $C\geq P(m,n)$ and \eqref{A leq Au leq mn 1/4}, we have
	\begin{align}\label{TraceRlargeGeneralSecondEstmt}
		\begin{split}
			\sqrt{ \tilde{m} \tilde{n}	}& \sum_{r_j\geq \max(\frac ax,1, P(m,n))} \left| \frac{\overline{\rho_j(m)} \rho_j(n)}{\ch \pi r_j}\widehat{\phi}(r_j)\right|\\
			&\ll \(\(\frac xT\)^{\frac12}+ (\tilde{m} +\tilde{n})^{\frac14}( \tilde{m}  \tilde{n})^{-\frac1{16}}A(m,n)^{\frac14}+ ( \tilde{m}  \tilde{n})^{\frac1{16}}A(m,n)^{\frac 34}\)(  \tilde{m}  \tilde{n}x)^\ep\\
			&\ll \( x^{\frac\delta2} + ( \tilde{m}  \tilde{n})^{\frac3{16}}A(m,n)^{\frac14}\) ( \tilde{m}  \tilde{n}x)^\ep.
		\end{split}
	\end{align}
Combining \eqref{TraceRlargeGeneralFirstPartEstmt} and \eqref{TraceRlargeGeneralSecondEstmt} we have
	\begin{equation}\label{TraceRlarge3General}
		\sqrt{ \tilde{m} \tilde{n}} \sum_{r_j\geq \max(\frac ax,1)} \left| \frac{\overline{\rho_j(m)} \rho_j(n)}{\ch \pi r_j}\widehat{\phi}(r_j)\right|\\
		\ll \( x^{\frac\delta2} + A_u(m,n)\) ( \tilde{m}  \tilde{n}x)^\ep. 
	\end{equation}
	
From \eqref{mainDiffrenceGEN}, \eqref{TraceSmoothingGEN}, \eqref{S2 weight 12}, \eqref{Trace estmt Holomorphic}, \eqref{GeneralTraceRisSmall}, \eqref{GeneralTraceRlarge1}, and \eqref{TraceRlarge3General}, we get
	\begin{align*}
		\sum_{\substack{x< c\leq 2x\\N|c}} \frac{S(m,n,c,\nu)}{c}-\sum_{r_j\in i (0,\frac14]}(2^{2s_j-1}-1)&\tau_j(m,n)\frac{x^{2s_j-1}}{2s_j-1}\\
		&\ll \(x^{\frac 12-\delta}+x^{\frac \delta 2}+A_u(m,n)\)(\tilde{m}\tilde{n}x)^\ep. 
	\end{align*}
	Proposition~\ref{PrereqPropGEN} follows by choosing $\delta=\frac13$. We finish the proof of Theorem~\ref{mainThm} in weight $\frac 12$.

	\subsection{On the case $k=-\frac12$}
	Recall the remark after Proposition~\ref{AhgDunEsstBd}. Let $\rho_j'(n)$ denote the Fourier coefficients of an orthonormal basis $\{v_j'(\cdot)\}$ of $\LEigenform_{\frac 32}(N,\nu)$. For each singular cusp $\mathfrak{a}$ of $(\Gamma,\nu)$, let $\Ea'(\cdot,s)$ be the associated Eisenstein series in weight $\frac 32$. Let $\rho_{\mathfrak{a}}'(n,r)$ be defined as in \eqref{FourierExpEsst} associated with $\Ea'(z,\frac12+ir)$ for $r\in \R$. 
	
	Recall the definition of the Maass lowering operator $L_k$ in \eqref{Maass Lowering op} and $H_\theta$ \eqref{HTheta} for $\theta=\frac7{64}$. By \eqref{LkBijIso}, the set
	\[\left\{v_j\defeq \(\tfrac1{16}+r_j^2\)^{-\frac12}L_{\frac 32}v_j':\ r_j\neq\tfrac i4\right\}
	\text{  is an orthonormal basis of  }
	\bigoplus_{r_j\neq \frac i4}\LEigenform_{-\frac12}(N,\nu,r_j). \] 
	Combining \cite[(4.36), (4.27) and the last equation of p. 502]{DFI2002} (which remain valid for $k\in \Z+\frac 12$), for $r_j\neq \frac i4$ and $\tilde n>0$, since
	\[L_{\frac 32} \(W_{\frac 34\tilde n,\,\im r}(4\pi \tilde ny)e(\tilde nx)\)=-(\tfrac 1{16}+r^2)W_{-\frac{\tilde n}4,\,\im r}(4\pi \tilde ny)e(\tilde nx), \] 
	the Fourier coefficient $\rho_j(n)$ of $v_j$ satisfies
	\begin{equation}\label{Weight -1/2 and 3/2 alternation - discrete, definition but not abs}
		\rho_j(n)=-(\tfrac1{16}+r^2)^{\frac12}\rho_j'(n)\qquad \text{for } r_j\neq \tfrac i4,\ \tilde n>0, 
	\end{equation}
	and then 
	\begin{equation}\label{Weight -1/2 and 3/2 alternation - discrete}
		|\rho_j(n)|\asymp|\rho_j'(n)| \quad \text{if }|r_j|\leq 1,\ \im r_j\leq \tfrac{\theta}2\qquad \text{and}\quad |\rho_j(n)|\asymp r|\rho_j'(n)|\quad \text{if }r_j\geq 1,
	\end{equation}
	where the bound $2\im r_j\leq \theta$ is from Proposition~\ref{specParaBoundWeightHalf}.

	In the case $r_j=\frac{i}{4}$, \eqref{CuspFormR0} and \eqref{Coeffi CuspFormR0} show that $\rho_j(n)=0$ and
	\begin{equation}\label{Weight -1/2 and 3/2 alternation - discrete bottom i/4}
	\tau_j(m,n)=0 \qquad \text{for\ }\tilde n>0,\ r_j=\tfrac i4. 
	\end{equation}
	
	Moreover, by \eqref{LkBijEisenstein}, if $\Ea(z,s)$ is the associated Eisenstein series in weight $-\frac12$, then
	\[L_{\frac 32}\Ea'(z,\tfrac12+ir)=(\tfrac14-ir)\Ea(z,s)\quad \text{and}\quad (\tfrac1{16}+r^2)^{\frac12}|\rho_{\mathfrak{a}}'(n,r)|=|\rho_{\mathfrak{a}}(n,r)|. \]
	We also get 
	\begin{equation}\label{Weight -1/2 and 3/2 alternation - continuous}
		|\rho_\mathfrak{a}(n,r)|\asymp|\rho_\mathfrak{a}'(n,r)| \quad \text{if }r\in[-1,1]\qquad \text{and}\quad |\rho_\mathfrak{a}(n,r)|\asymp r|\rho_\mathfrak{a}'(n,r)|\quad \text{if }|r|\geq 1. 
	\end{equation}
	
	We have the following proposition: 
	\begin{proposition}\label{PrereqPropGEN weight 1.5}
		With the same setting as Theorem~\ref{mainThm} for $k=-\frac12$, when $2x\geq A_u(m,n)^2$, we have 
		\begin{align*}
			\begin{split}
				\sum_{\substack{x< c\leq 2x\\N|c}} \frac{S(m,n,c,\nu)}{c} -\sum_{r_j\in i (0,\frac \theta2]}(2^{2s_j-1}-1)\tau_j(m,n)\frac{x^{2s_j-1}}{2s_j-1}\ll \(x^{\frac 16}+A_u(m,n)\)(\tilde{m}\tilde{n}x)^\ep,
			\end{split}
		\end{align*}
	\end{proposition}
	Note that here $\tau_j(m,n)$ is defined in weight $-\frac12$, i.e.
	\[\tau_j(m,n)=2i^{-\frac12}\overline{\rho_j(m)}\rho_j(n)\pi^{1-2s_j}(4\tilde m \tilde n)^{1-s_j}\frac{\Gamma(s_j-\frac14)\Gamma(2s_j-1)}{\Gamma(s_j+\frac 14)}\]
	where $\rho_j(n)$ is from \eqref{Weight -1/2 and 3/2 alternation - discrete} as the Fourier coefficient of $v_j\in \LEigenform_{-\frac12}(N,\nu, r_j)$.  
	
	The proof that Proposition~\ref{PrereqPropGEN weight 1.5} implies Theorem~\ref{mainThm} in the case $k=-\frac12$ is the same as the case of weight $\frac12$ before. This is because $\tau_j(m,n)=0$ for $r_j=\frac i4$ \eqref{Weight -1/2 and 3/2 alternation - discrete bottom i/4} and because \eqref{Bound for tau_j}, \eqref{taujestGEN} and \eqref{differenceBetweenEstmtedTauAndTrueTauGEN} still hold for $r_j\in i(0,\frac \theta2]$ (the process only involves estimates on $\rho_j(n)$ with some applications of Proposition~\ref{AhgDunEsstBd} in weight $-\frac12$). In the rest of this subsection we prove Proposition~\ref{PrereqPropGEN weight 1.5}. 
	
	First we show that the main terms corresponding to $r_j=it_j\in i(0,\frac \theta 2]$ are the same when we shift the weight between $-\frac12$ and $\frac32$. Recall $s_j=\frac12+t_j$. Let $\tau_j'(m,n)$ denote the corresponding coefficients for $x^{2s_j-1}$ in weight $\frac 32$: 
	\begin{equation*}
		\tau_j'(m,n)=2e^{\frac{3\pi i}4}\overline{\rho_j'(m)}\rho_j'(n)\pi^{-2t_j}(4\tilde m\tilde n)^{\frac12-t_j}\frac{\Gamma(\frac 54+t_j)\Gamma(2t_j)}{\Gamma(t_j-\frac14)},
	\end{equation*}
	where $\rho_j'(n)$ is defined at the beginning of this subsection. 
	
	We claim that 
	\begin{equation}
		\tau_j'(m,n)=\tau_j(m,n), \quad \text{for }\tilde m,\tilde n>0 \text{ and }r_j\in i(0,\tfrac14]. 
	\end{equation}
When $r_j=\frac i4$, this is true because both of them equal to zero by \eqref{Weight -1/2 and 3/2 alternation - discrete bottom i/4} and $\Gamma(0)=\infty$.
When $r_j\in i(0,\frac \theta 2]$,
	\begin{align*}
		\tau_j(m,n)&=2e^{-\frac{\pi i}4}\overline{\rho_j(m)}\rho_j(n)\pi^{-2t_j}(4\tilde m\tilde n)^{\frac12-t_j}\frac{\Gamma(\frac 14+t_j)\Gamma(2t_j)}{\Gamma(\frac 34+t_j)}\\
		&= -2e^{\frac{3\pi i}4}\(\frac1{16}-t_j^2\)\overline{\rho_j'(m)}\rho_j'(n)\pi^{-2t_j}(4\tilde m\tilde n)^{\frac12-t_j}\frac{\Gamma(\frac 54+t_j)/(\frac14+t_j)}{(-\frac14+t_j)\Gamma(-\frac 14+t_j)}\Gamma(2t_j)\\
		&=\tau_j'(m,n). 
	\end{align*}

	Recall that the definition on $\widehat{\phi}$ \eqref{widehatPhi} is for weight $k\geq 0$ and here we use $\widehat{\phi}$ for weight $\frac 32$. We derive
	\begin{equation}\label{taujReference5.4GEN weight 1.5}
			4\sqrt{\tilde{m}\tilde{n}}\;\frac{\overline{\rho_j'(m)}\rho_j'(n)}{\ch \pi r_j}\widehat{\phi}(r_j)
			=(2^{2s_j-1}-1)\tau_j'(m,n)\frac{x^{2s_j-1}}{2s_j-1}+O\(A_u(m,n)(\tilde{m}\tilde{n})^\ep\). \end{equation}
	by the same process as we derive \eqref{taujReference5.4GEN} above. Since $\tau_j'(m,n)=0$ when $r_j=\frac i4$, we have $2t_j\leq \theta<\delta$ (with $\theta=\frac 7{64}$ \eqref{HTheta} and $\delta=\frac13$ chosen later) by Proposition~\ref{specParaBoundWeightHalf} and still get
	\begin{align}\label{mainDiffrenceGEN weight 1.5}
		\begin{split}
			&\left|\sum_{\substack{x<c\leq 2x\\N|c}} \frac{S(m,n,c,\nu)}{c} -\sum_{r_j\in i (0,\frac\theta2]}(2^{2s_j-1}-1)\tau_j'(m,n)\frac{x^{2s_j-1}}{2s_j-1} \right|\\
			\leq& \left|\sum_{\substack{x<c\leq 2x\\N|c}} \frac{S(m,n,c,\nu)}{c} -\sum_{N|c>0} \frac{S(m,n,c,\nu)}{c}\phi\(\frac ac\)\right| +O\(A_u(m,n)(\tilde{m}\tilde{n})^{\ep}\)\\
			&+ \left|\sum_{N|c>0} \frac{S(m,n,c,\nu)}{c}\phi\(\frac ac\)-4\sqrt{\tilde{m}\tilde{n}}\sum_{r_j\in i (0,\frac\theta2]}\frac{\overline{\rho_j'(m)}\rho_j'(n)}{\ch \pi r_j}\widehat{\phi}(r_j)\right|\\
			=:&\;S_3+O\(A_u(m,n)(\tilde{m}\tilde{n})^{\ep}\)+S_4. 
		\end{split}
	\end{align}
	The first sum $S_3$ above can be estimated similarly by condition (2) of Definition \ref{Admissibility} as
	\begin{align}\label{TraceSmoothingGEN weight 1.5}
		\begin{split}
			S_3\leq\sum_{\substack{x-T\leq c\leq x\\2x\leq c\leq 2x+2T\\N|c}}\frac{|S(m,n,c,\nu)|}{c}\ll_{N,\nu,\delta,\ep} x^{\frac12-\delta}(\tilde{m}  \tilde{n}x)^\ep. 
		\end{split}
	\end{align}
	By Theorem~\ref{traceFormula Theorem}, 
	\[S_4\ll |\mathcal U_\frac 32|+\la \sqrt{\tilde m\tilde n}\sum_{r_j\geq 0}\frac{\overline{\rho_j'(m)}\rho_j'(n)}{\ch \pi r_j} \widehat{\phi}(r_j)+\sqrt{\tilde m\tilde n}\sum_{\mathrm{singular\ }\mathfrak{a}}\int_{-\infty}^\infty \overline{\rho_\mathfrak{a}'(m,r)}\rho_\mathfrak{a}'(n,r)\frac{\widehat{\phi}(r)}{\ch \pi r }dr\ra. \]
	The bound for $\mathcal{U}_{\frac 32}$ is done in \eqref{Trace estmt Holomorphic}. Estimates for the remaining part of $S_4$ follow from the same process as \S\ref{5.1.2} in the case of weight $\frac12$, taking \eqref{Weight -1/2 and 3/2 alternation - discrete} and \eqref{Weight -1/2 and 3/2 alternation - continuous} into account. For the same reason as the beginning of \S\ref{5.1.2}, we just record the bounds with respect to the discrete spectrum here. 
	
	For $r\in [0,1)$, we apply Proposition~\ref{AhgDunEsstBd}, \eqref{Weight -1/2 and 3/2 alternation - discrete} and \eqref{RisReal 0 to 1} to get
	\begin{equation}\label{GeneralTraceRisSmall weight 1.5}
		\sqrt{ \tilde{m}  \tilde{n}}\sum_{r\in [0,1)} \left|\frac{\overline{\rho_j'(m)}\rho_j'(n)}{\ch \pi r_j} \widehat{\phi}(r_j)\right| \ll \sqrt{ \tilde{m}  \tilde{n}}\sum_{r\in [0,1)} \left|\frac{\overline{\rho_j(m)}\rho_j(n)}{\ch \pi r_j} \widehat{\phi}(r_j)\right| 
		\ll  A(m,n) ( \tilde{m}  \tilde{n}x)^\ep.
	\end{equation}
	
	For $r\in[1,\frac ax)$, we apply Proposition~\ref{AhgDunEsstBd}, $\rho_j'(n)\ll r_j^{-1}|\rho_j(n)|$ from \eqref{Weight -1/2 and 3/2 alternation - discrete}, and  $\widehat{\phi}(r)\ll 1$ from \eqref{Rlarge}.  
	Since
	\begin{align}
		\begin{split}
			s(R)\defeq \sqrt{ \tilde{m}  \tilde{n}}\sum_{r\in [1,R]} \left|\frac{\overline{\rho_j(m)}\rho_j(n)}{\ch \pi r_j} \right|
			\ll   A(m,n)R^{\frac 72}(\tilde{m} \tilde{n})^{\ep}
		\end{split}
	\end{align}
	by Cauchy-Schwarz, with the help of \eqref{Pow ax A leq Au when x large} we have
	\begin{align}\label{GeneralTraceRlarge1 weight 1.5}
		\begin{split}
			\sqrt{ \tilde{m}  \tilde{n}}\sum_{r_j\in [1,\frac a{x})} \left|\frac{\overline{\rho_j'(m)}\rho_j'(n)}{\ch \pi r_j} \widehat{\phi}(r_j)\right|
			&\ll\sqrt{ \tilde{m}  \tilde{n}}\sum_{r_j\in [1,\frac a{x})} \left|\frac{\overline{\rho_j(m)}\rho_j(n)}{\ch \pi r_j}r_j^{-2}\right|\\
			&\ll r^{-2}s(r)\Big|_{r=1}^{\frac a{x}}+\int_1^{\frac{a}{x}}s(r) r^{-3}dr\\
			&\ll A(m,n) \(\frac ax\)^{\frac 32}(  \tilde{m}  \tilde{n}x)^\ep \\
			&\ll A_u(m,n)(  \tilde{m}  \tilde{n}x)^\ep.
		\end{split}
	\end{align}

	We still let 
	\[P(m,n)= 2(\tilde{m}\tilde{n})^{\frac18}A(m,n)^{-\frac12}\geq 1\] 
	and divide $r\geq \max(\frac ax,1)$ into two parts: 
	$\max\(\frac ax,1\)\leq r< P(m,n)$
	and 
	$r \geq  \max\(\frac ax,1,P(m,n)\)$. 
	In the first range, we apply Proposition~\ref{AhgDunEsstBd}, \eqref{Weight -1/2 and 3/2 alternation - discrete} and $\widehat{\phi}(r)\ll 1$ from \eqref{Rlarge} to get
	\begin{equation}\label{TraceRlargeGeneralFirstPartEstmt weight 1.5}
		\sqrt{ \tilde{m} \tilde{n}} \sum_{\max(\frac ax,1)\leq r_j<P(m,n)} \left| \frac{\overline{\rho_j'(m)} \rho_j'(n)}{\ch \pi r_j}\widehat{\phi}(r_j)\right|\\
		\ll A_u(m,n)(  \tilde{m}  \tilde{n}x)^\ep
	\end{equation}
	by partial summation similar as \eqref{GeneralTraceRlarge1 weight 1.5}. We divide the second range into dyadic intervals $C\leq r_j< 2C$ and apply Proposition~\ref{AdsDukeEsstBd}, \eqref{Weight -1/2 and 3/2 alternation - discrete} and $\widehat{\phi}(r)\ll \min(1,\frac x{rT})$ from \eqref{Rlarge}: 
	\begin{align}\label{TraceRlargeGeneralSecondEstmtDyadic weight 1.5}
		\begin{split}
			\sqrt{ \tilde{m}\tilde{n}}& \sum_{C\leq r_j< 2C} \left| \frac{\overline{\rho_j'(m)} \rho_j'(n)}{\ch \pi r_j}\widehat{\phi}(r_j)\right|\ll \sqrt{ \tilde{m}\tilde{n}} \sum_{C\leq r_j< 2C} \left| \frac{\overline{\rho_j(m)} \rho_j(n)}{\ch \pi r_j} r_j^{-2}\widehat{\phi}(r_j)\right|\\
			&\ll \(\min\(C^{\frac12},C^{-\frac 12}\frac xT\)+ (\tilde{m}^{\frac14} +\tilde{n}^{\frac14})C^{-\frac12}+ (\tilde{m}  \tilde{n})^{\frac14}C^{-\frac32}\)( \tilde{m}  \tilde{n}x)^\ep.
		\end{split}
	\end{align}
	Summing up from \eqref{TraceRlargeGeneralSecondEstmtDyadic weight 1.5} similar as we did after \eqref{TraceRlargeGeneralSecondEstmtDyadic} and recalling $T\asymp x^{1-\delta}$ in Setting~\ref{conditionATDelta}, we have
	\begin{align}\label{TraceRlargeGeneralSecondEstmt weight 1.5}
		\begin{split}
			\sqrt{ \tilde{m} \tilde{n}	} \sum_{r_j\geq \max(\frac ax,1, P(m,n))} \left| \frac{\overline{\rho_j'(m)} \rho_j'(n)}{\ch \pi r_j}\widehat{\phi}(r_j)\right|
			\ll \( x^{\frac\delta2} + ( \tilde{m}  \tilde{n})^{\frac3{16}}A(m,n)^{\frac14}\) (\tilde{m}  \tilde{n}x)^\ep.
		\end{split}
	\end{align}
	From \eqref{TraceRlargeGeneralFirstPartEstmt weight 1.5} and \eqref{TraceRlargeGeneralSecondEstmt weight 1.5} we have
	\begin{equation}\label{TraceRlarge3General weight 1.5}
		\sqrt{ \tilde{m} \tilde{n}} \sum_{r_j\geq \max(\frac ax,1)} \left| \frac{\overline{\rho_j'(m)} \rho_j'(n)}{\ch \pi r_j}\widehat{\phi}(r_j)\right|\\
		\ll \( x^{\frac\delta2} + A_u(m,n)\) ( \tilde{m}  \tilde{n}x)^\ep. 
	\end{equation}
	
	Combining \eqref{mainDiffrenceGEN weight 1.5}, \eqref{TraceSmoothingGEN weight 1.5}, \eqref{Trace estmt Holomorphic}, \eqref{GeneralTraceRisSmall weight 1.5}, \eqref{GeneralTraceRlarge1 weight 1.5}, and \eqref{TraceRlarge3General weight 1.5}, we get
	\begin{align*}
		\sum_{\substack{x< c\leq 2x\\N|c}} \frac{S(m,n,c,\nu)}{c}-\sum_{r_j\in i (0,\frac14]}(2^{2s_j-1}-1)&\tau_j(m,n)\frac{x^{2s_j-1}}{2s_j-1}\\
		&\ll \(x^{\frac 12-\delta}+x^{\frac \delta 2}+A_u(m,n)\)(\tilde{m}\tilde{n}x)^\ep. 
	\end{align*}
	Proposition~\ref{PrereqPropGEN weight 1.5} follows by choosing $\delta=\frac13$ and we finish the proof of Theorem~\ref{mainThm}. 
	
	\begin{proof}[Proof of Theorem~\ref{mainThmLastSecTheEquation}]
		The proof follows from the same process as \cite[\S 9.2]{QihangFirstAsympt}. Note that we need to restrict $\sum_{r_j=\frac i4}\tau_j(m,n)=0$ when $\tilde m>0$, $\tilde n>0$ and $k=\frac 12$ (and the conjugate case $\tilde m<0$, $\tilde n<0$ and $k=-\frac 12$ by \eqref{KlstmSumConj}), otherwise the sum may not converge. 
	\end{proof}	
	
	\section*{Acknowledgement}
	The author thanks the referee for the careful reading and helpful comments. The author also thanks Professor Scott Ahlgren for his plenty of helpful discussions and suggestions.

	\bibliographystyle{alpha}
	\bibliography{allrefs}

\end{document}